\documentclass{article} 

\usepackage{graphicx}

\usepackage{pinlabel}

\usepackage{amsthm}

\newcommand{\pic}[2]{\raisebox{-.5\height}{\includegraphics[scale=#2]{#1}}}
\newcommand{\pica}[2]{\raisebox{-.6\height}{\includegraphics[scale=#2]{#1}}}
\newcommand{\picb}[2]{\raisebox{-.3\height}{\includegraphics[scale=#2]{#1}}}
\newcommand{\picc}[2]{\raisebox{-.4\height}{\includegraphics[scale=#2]{#1}}}

\def\AB{\pic{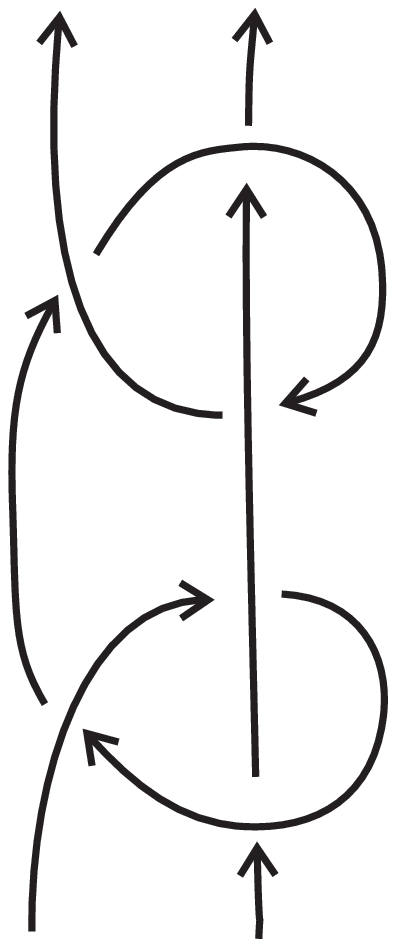}{.300}}

\def\ABreef{\pic{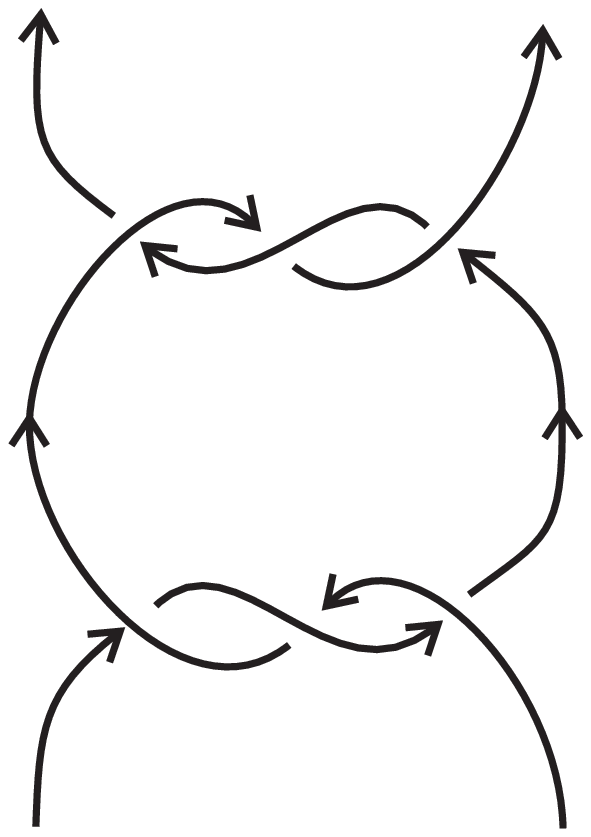}{.300}}

\def\Xor{\pic{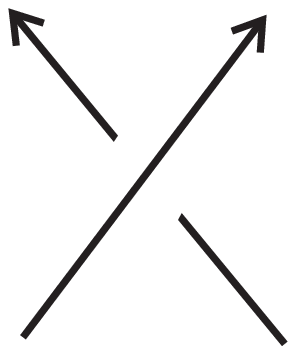}{.200}}
\def\Yor{\pic{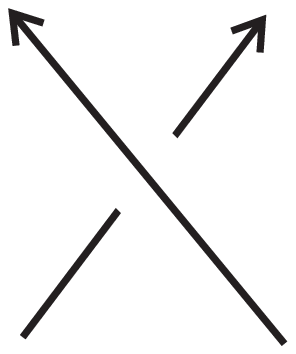} {.200}}
\def\Ior{\pic{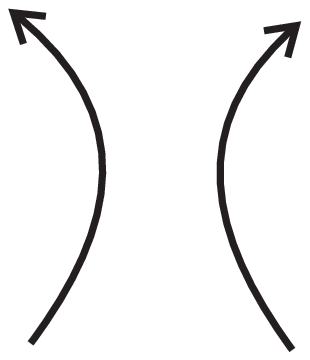} {.200}}

\def\Satellite{\pic{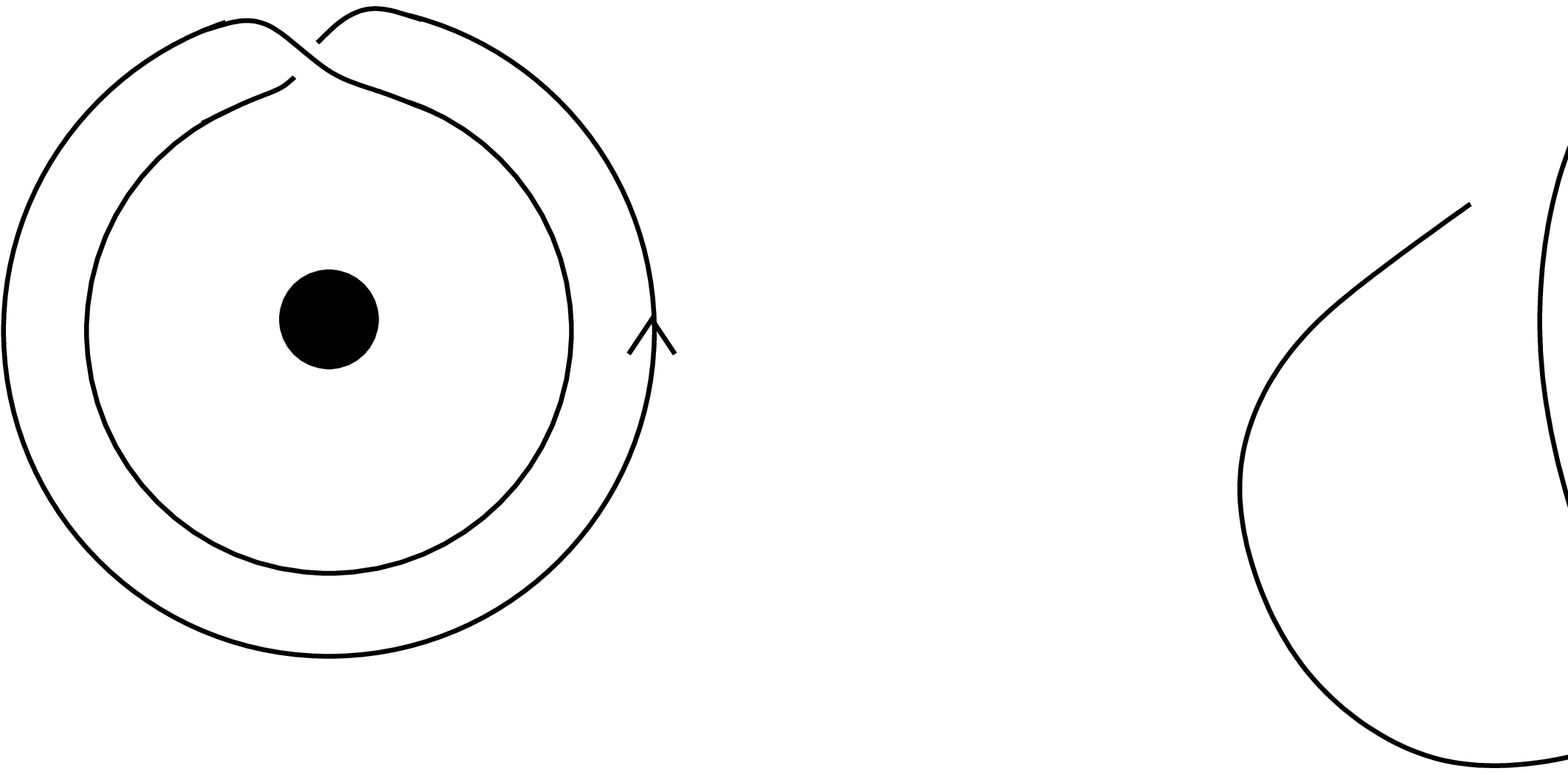}{.200}}

\def\Idor{\pic{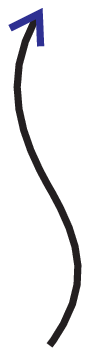} {.200}}
\def\Rcurlor{\pic{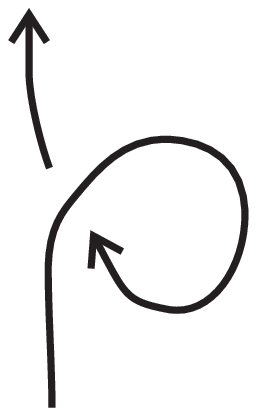} {.200}}
\def\Lcurlor{\pic{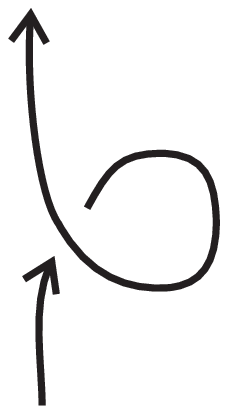} {.20}}
\def\Conway{\pic{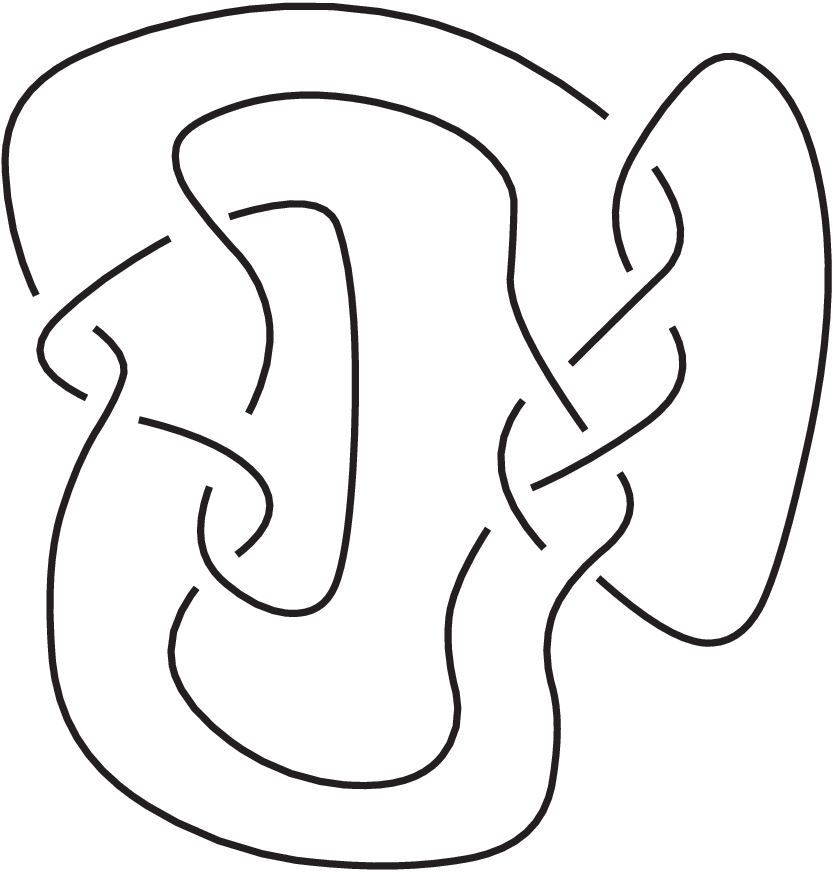} {  .40}}
\def\KT{\pic{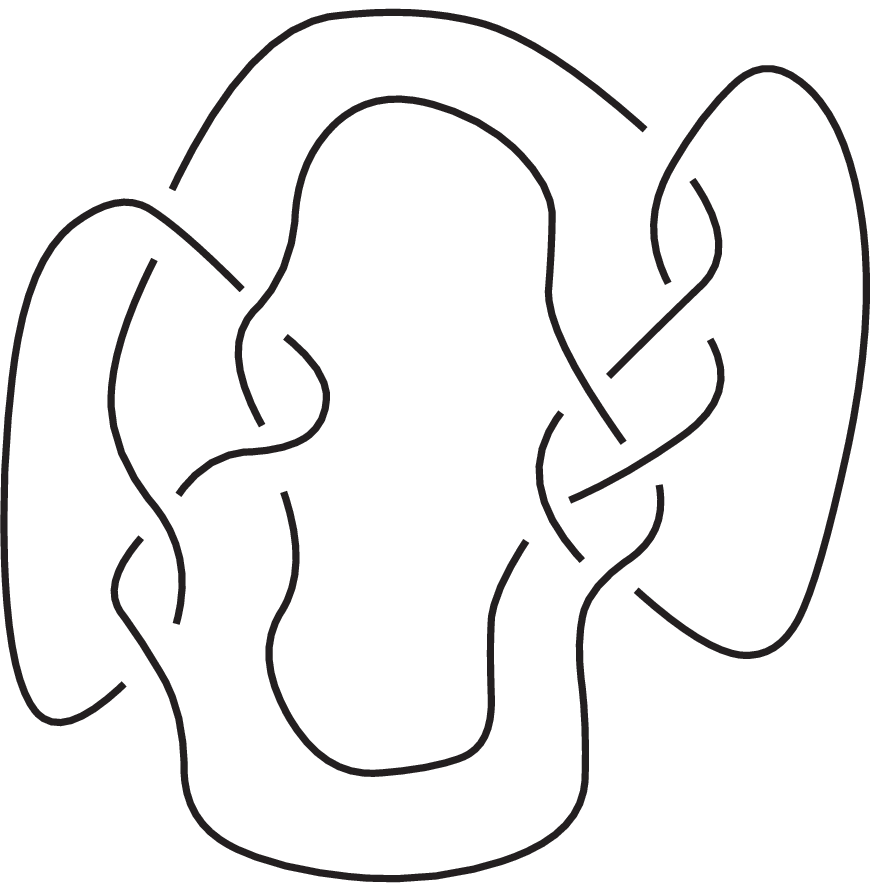} {.40}}

\def\mutantone{\pica{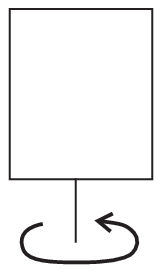} {.6}}
\def\mutanttwo{\pic{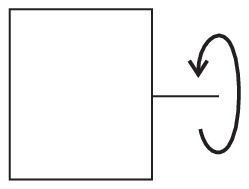} {.6}}
\def\mutantthree{\pic{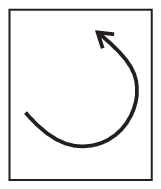} { .6}}

\def\mutantbox{\pic{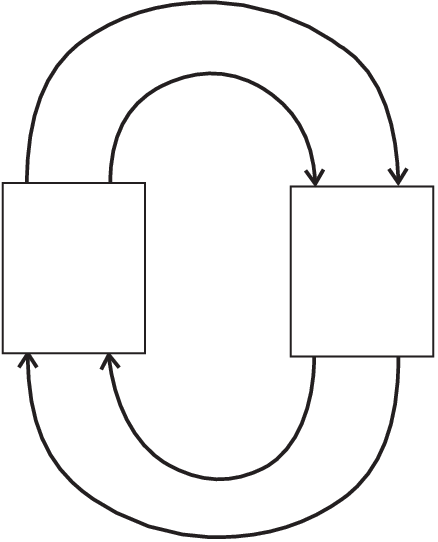}{ .6}}

\def\mutantboxrev{\pic{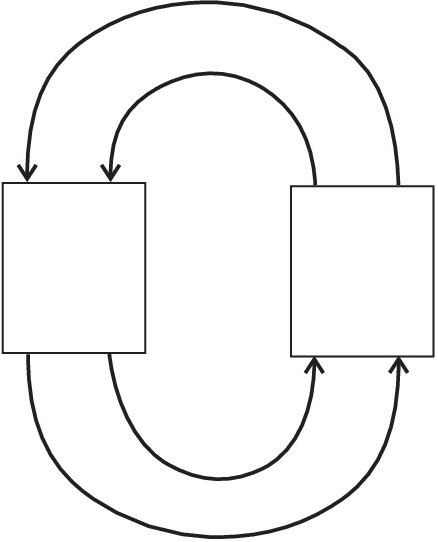}{ .6}}

\def\Sfiftyfive{\pic{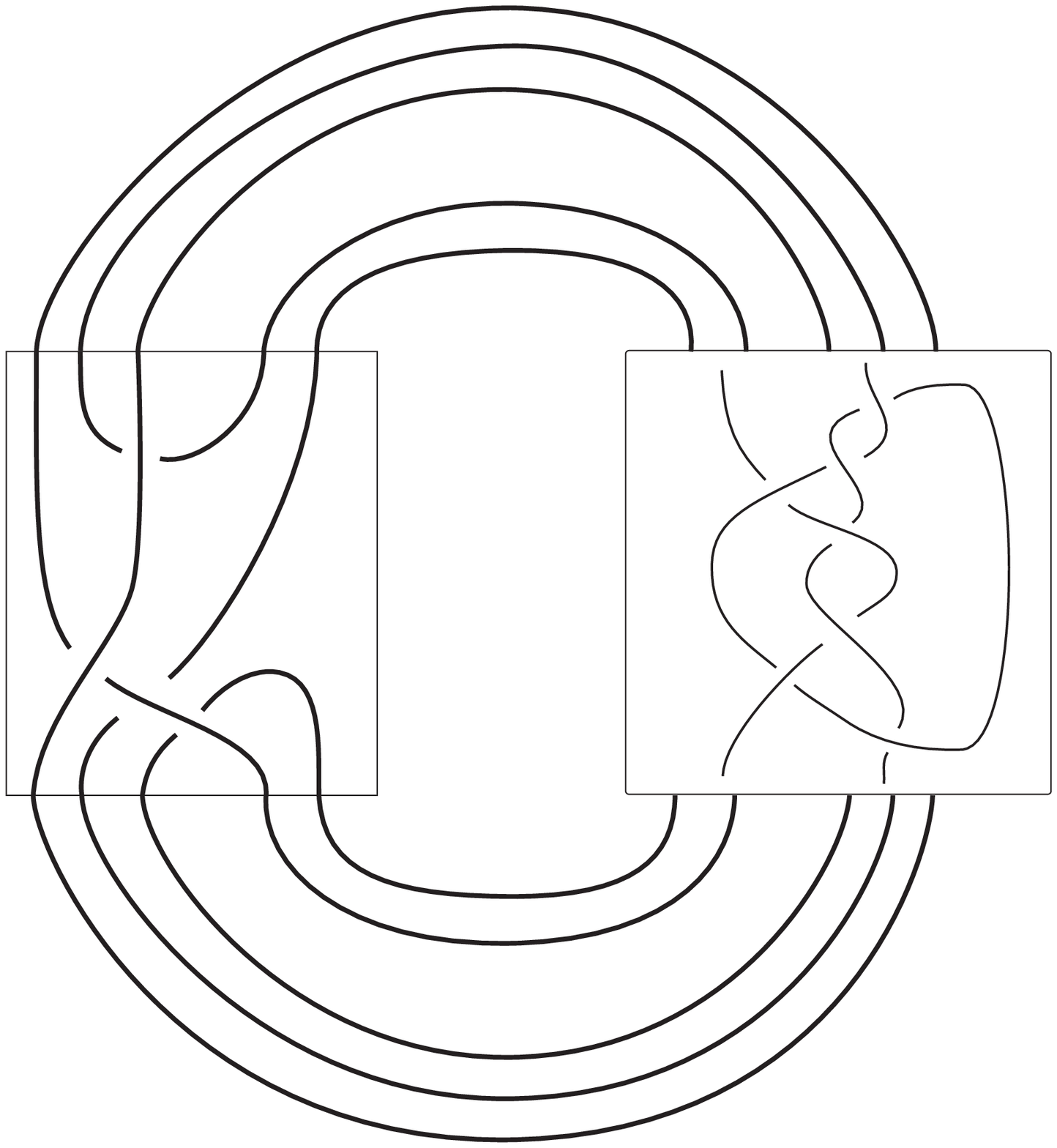}{.250}}
\def\Tfiftyfive{\pic{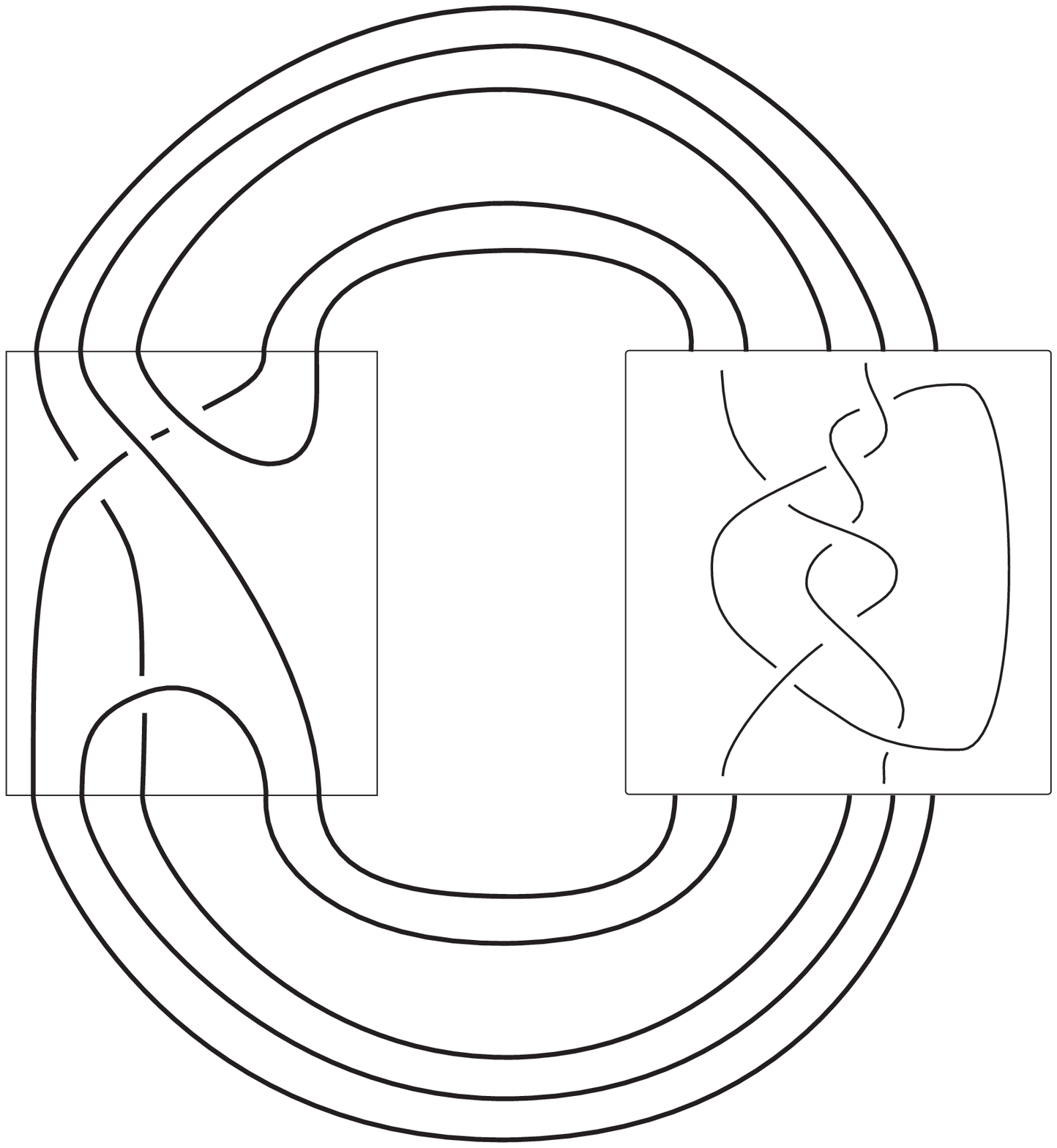}{.250}}

\def\Sfiftysix{\pic{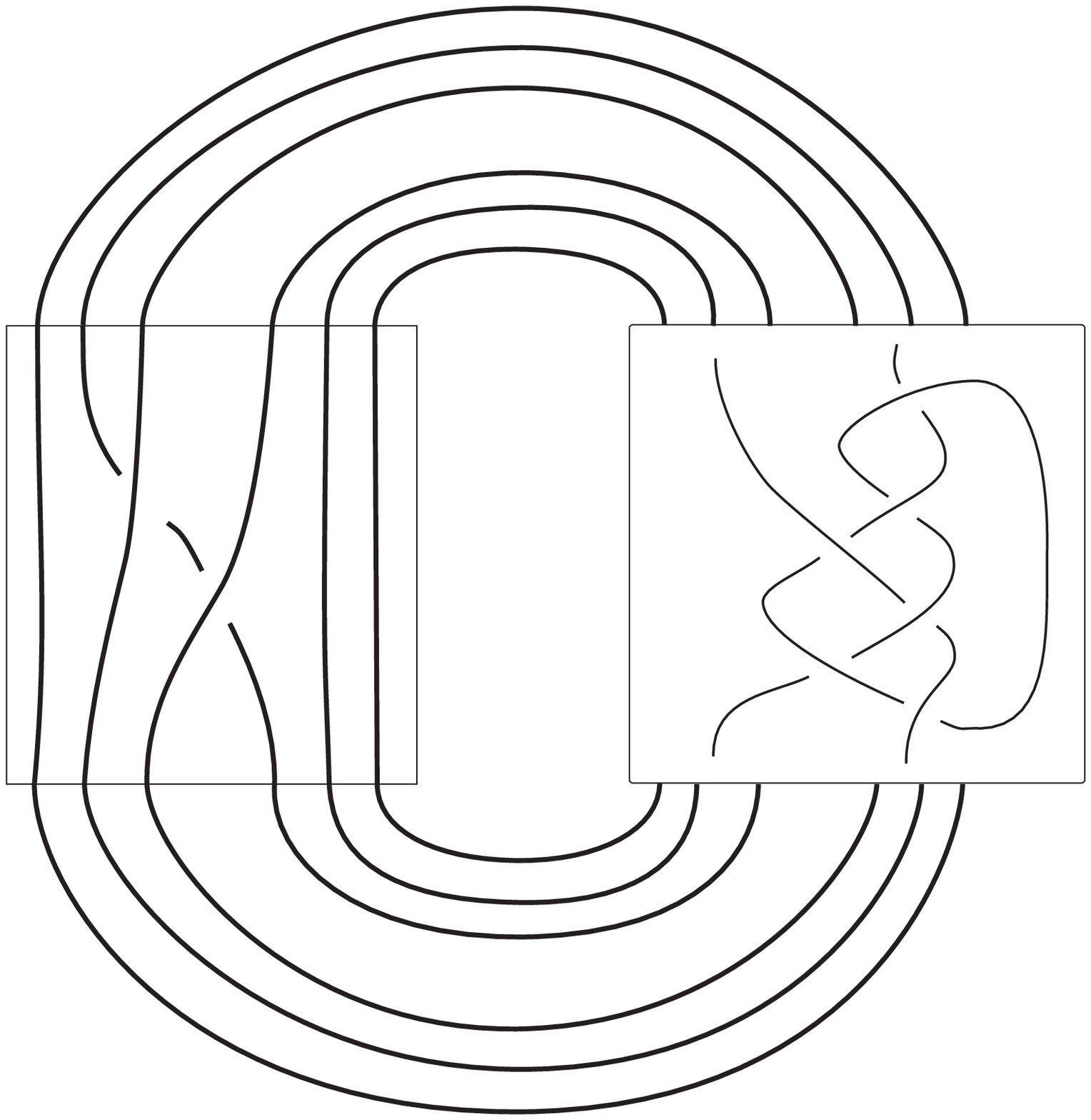}{.250}}
\def\Tfiftysix{\pic{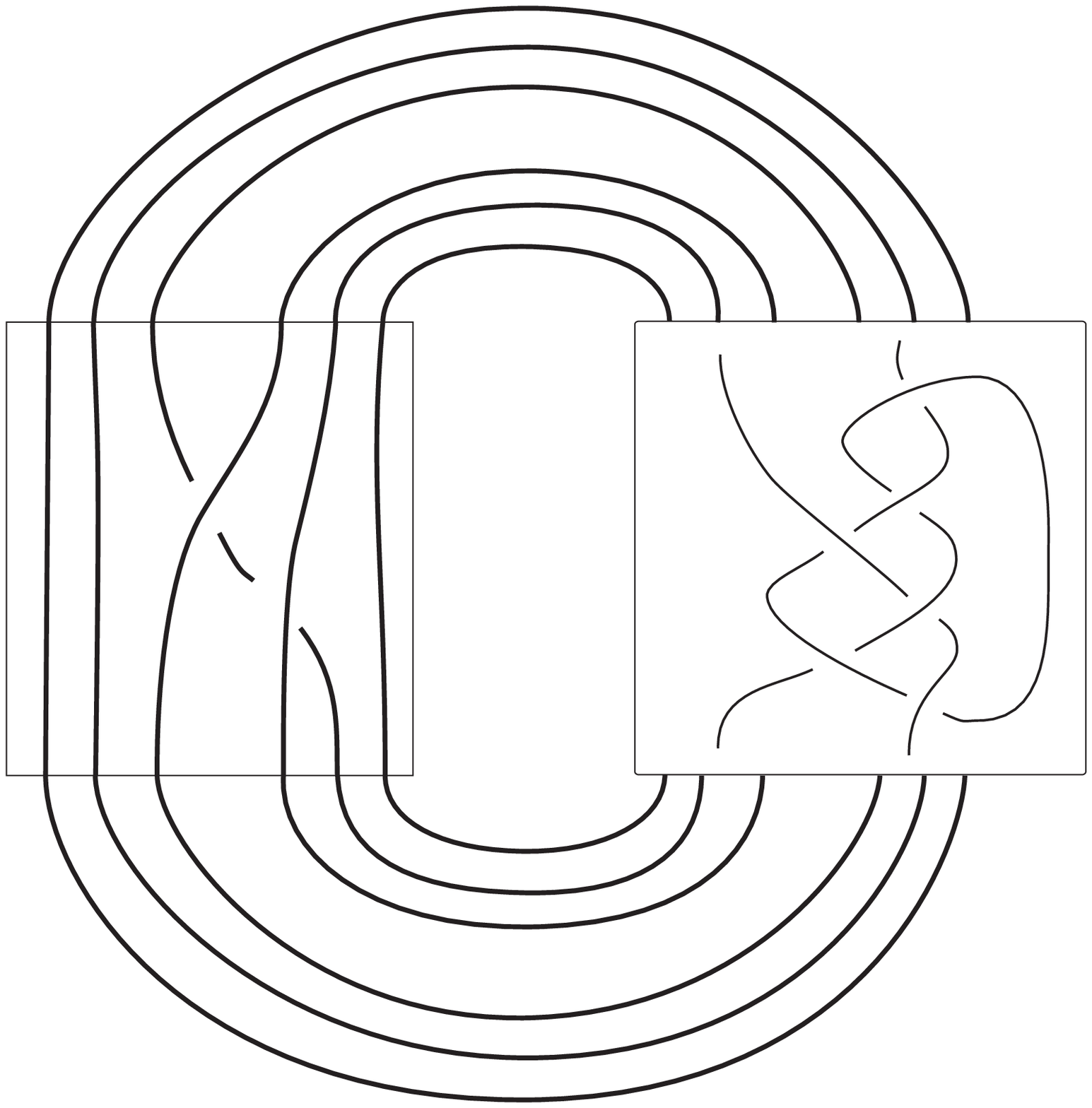}{.250}}

\def\parallels{\pic{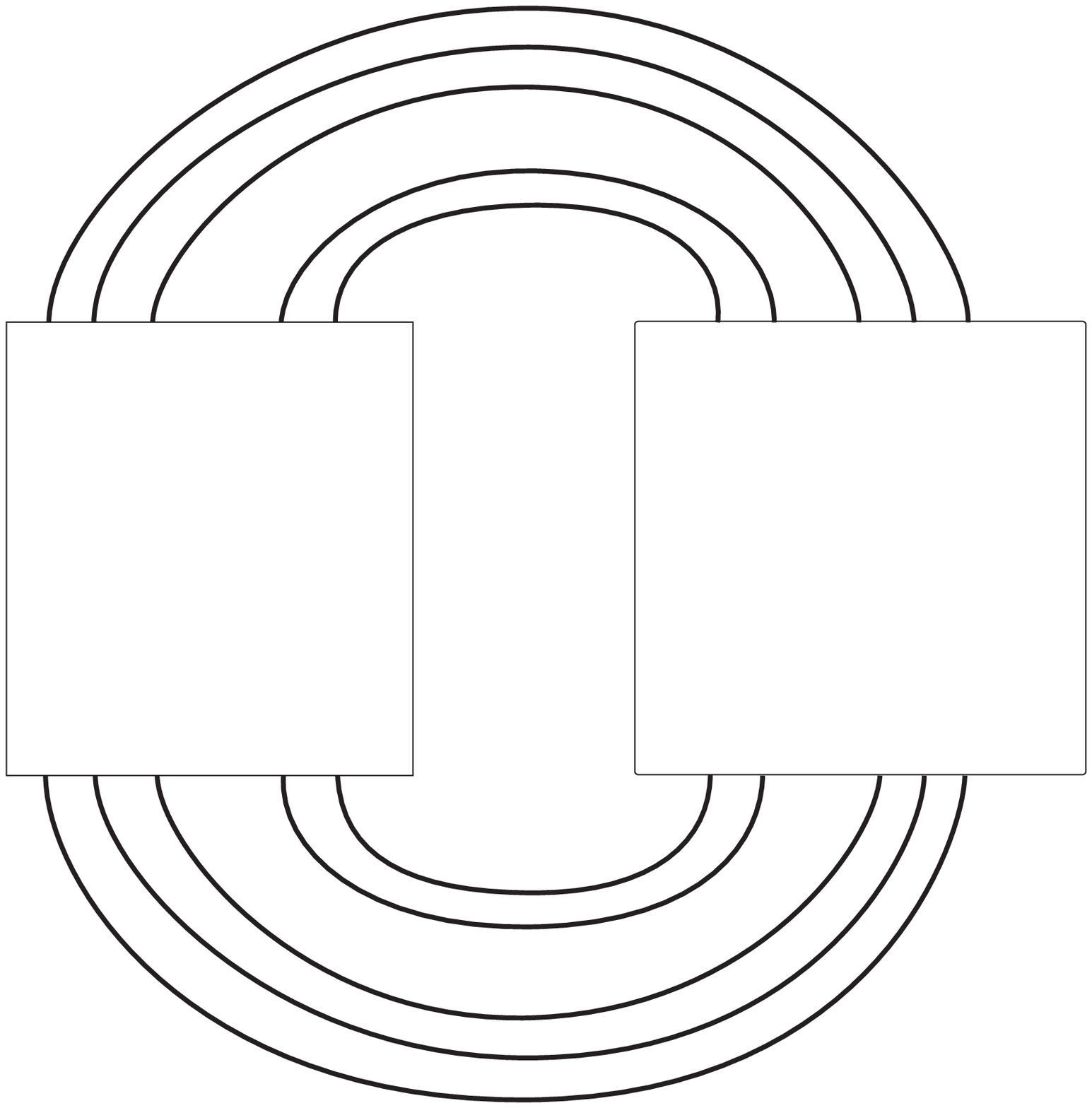}{.250}}
\def\crossparallels{\pic{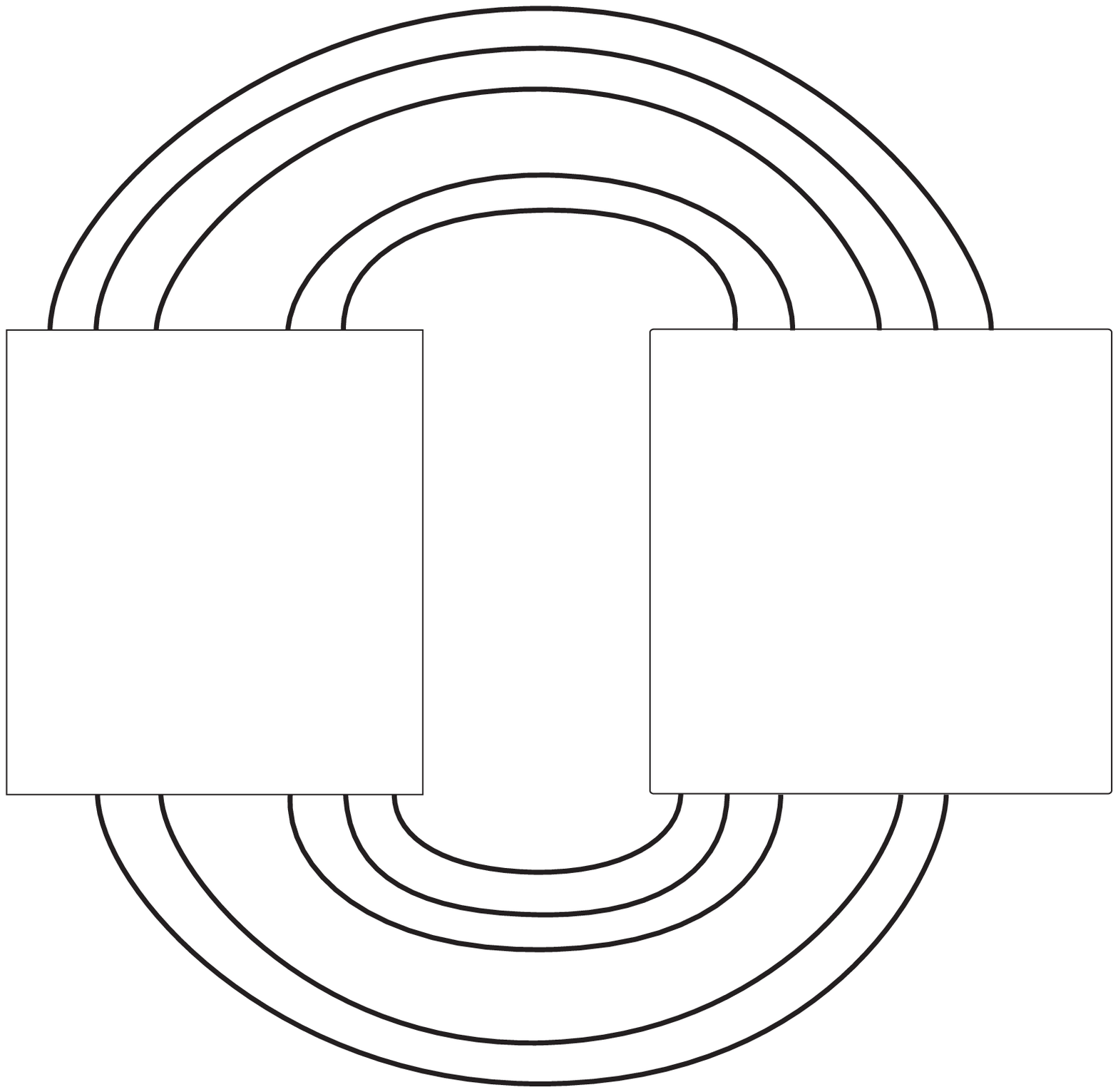}{.250}}

\def\unknotunor{\picb{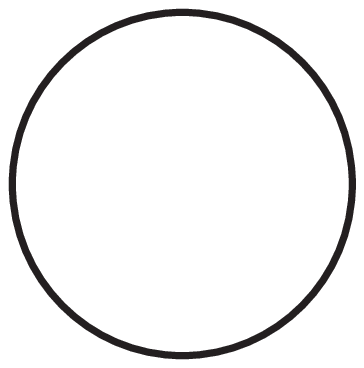} {.150}}

\def\xunor{\picc{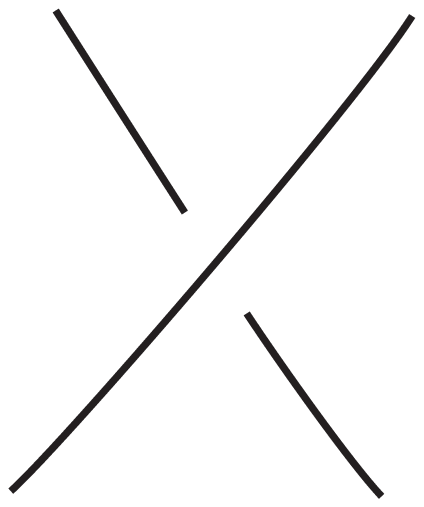} {.15}}

\def\iunor{\picc{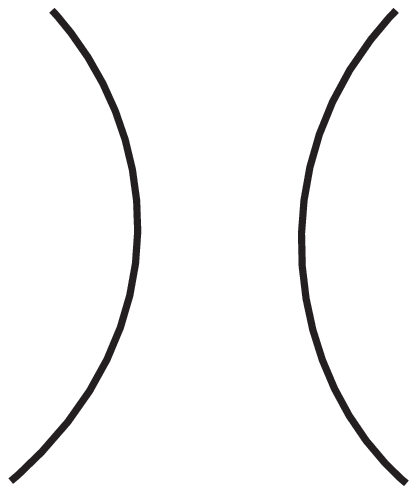} {.15}}
\def\idunor{\picb{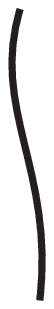} {.2}}
\def\infunor{\picb{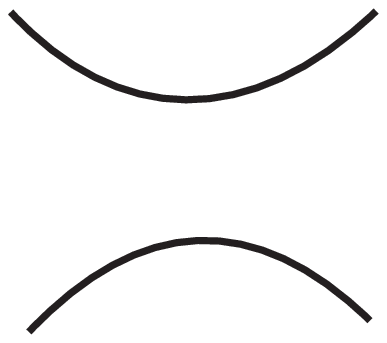} {.15}}
\def\iddeltaunor{\picb{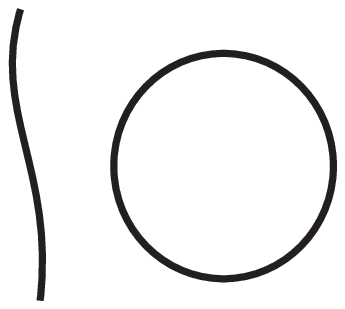} {.2}}

\def\Braid{\pic{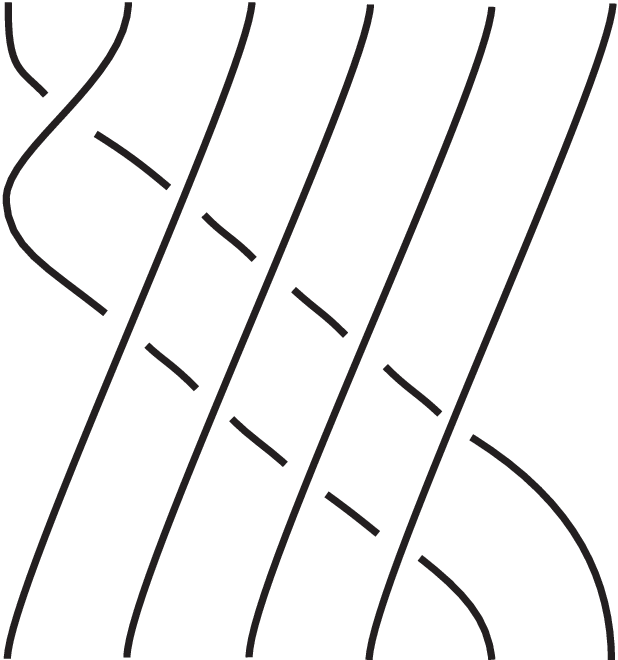} {.30}}

\def\twicepunc{\pic{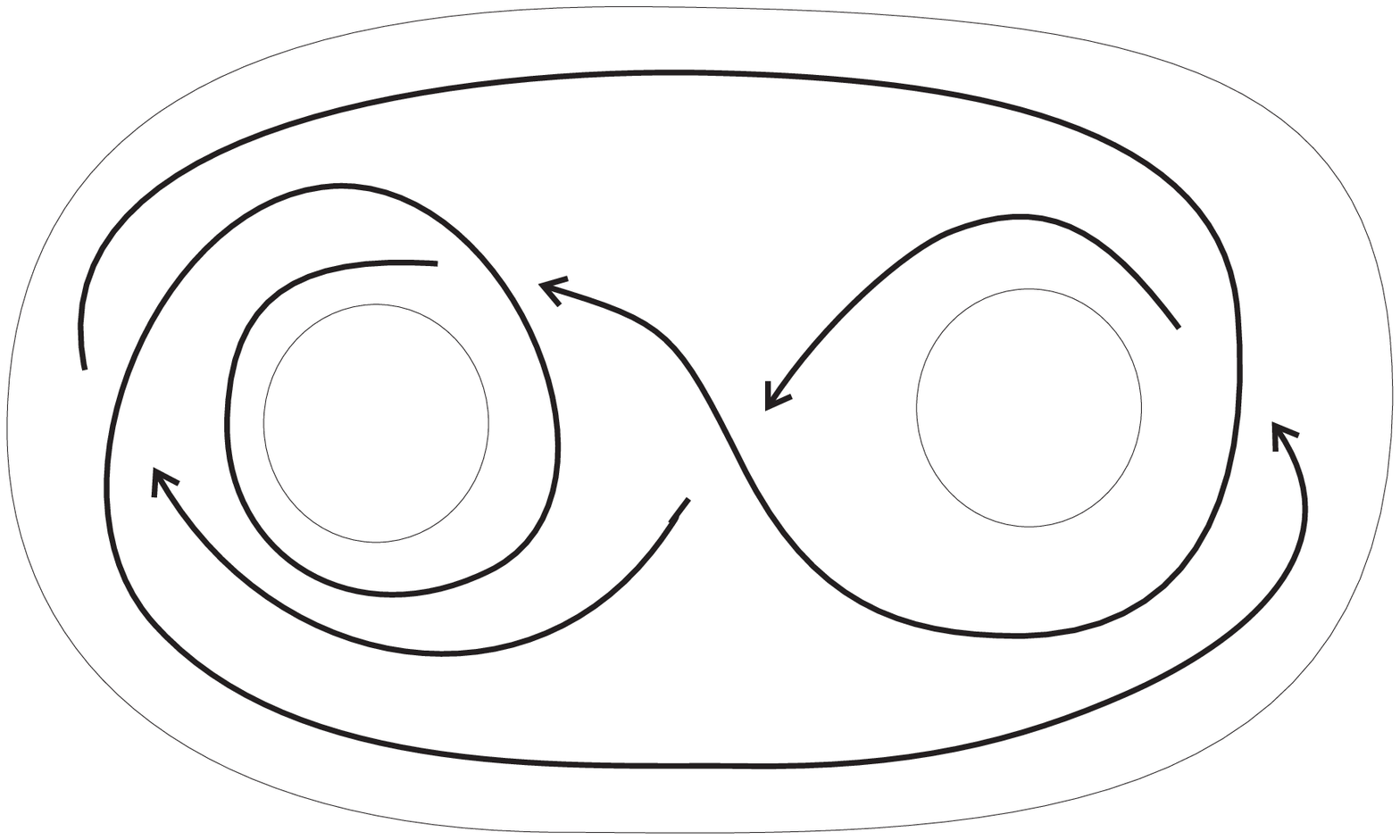} {.250}}

\def\Ctangle{\pic{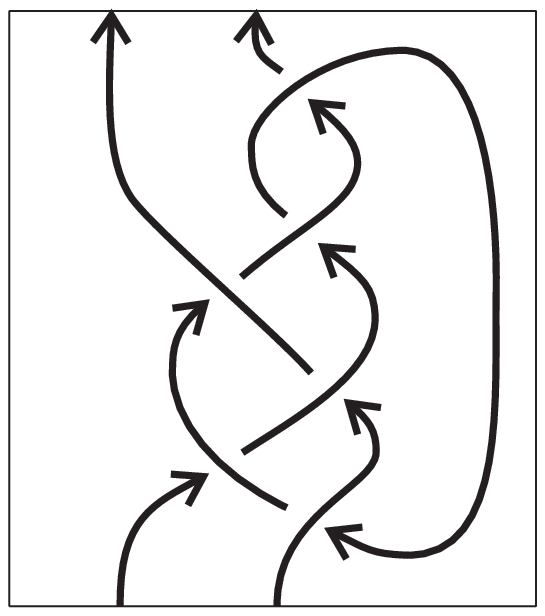} {.30}}

\def\squarepants{\pic{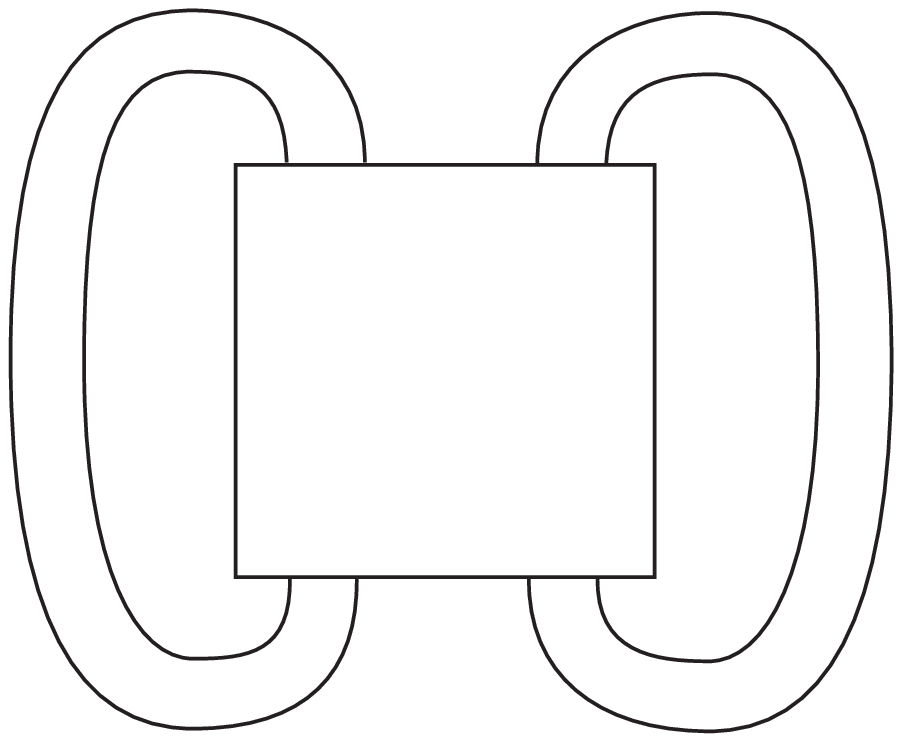} {.30}}
\def\squaretorus{\pic{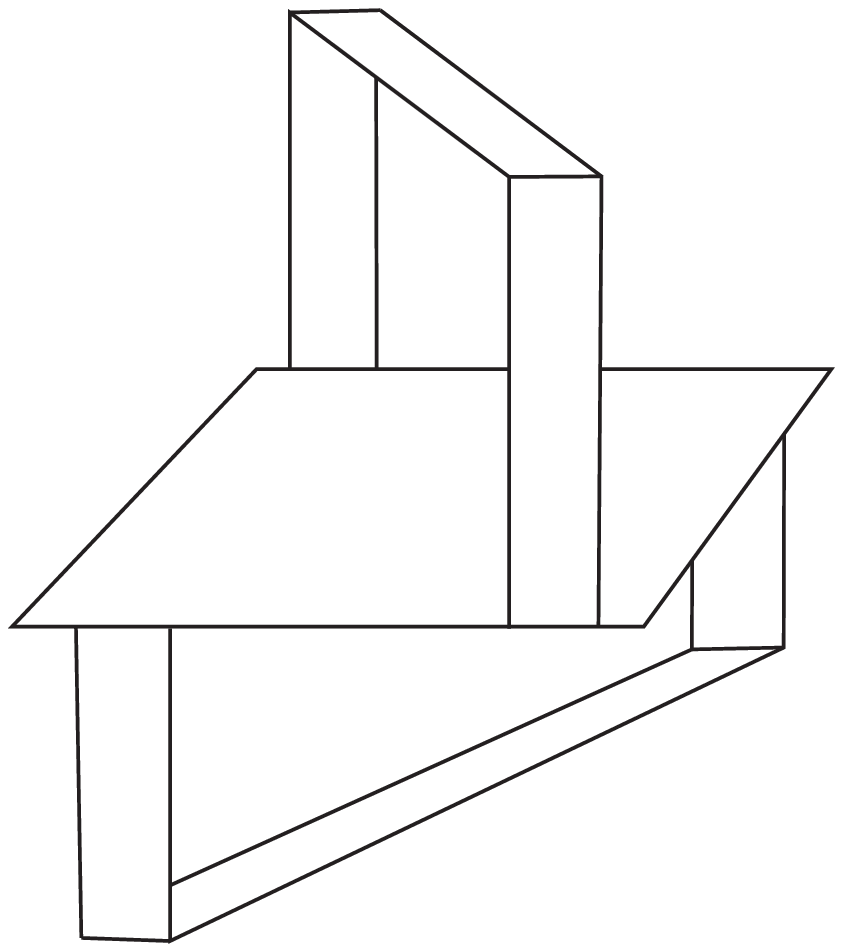} {.350}}

\def\torushole{\pic{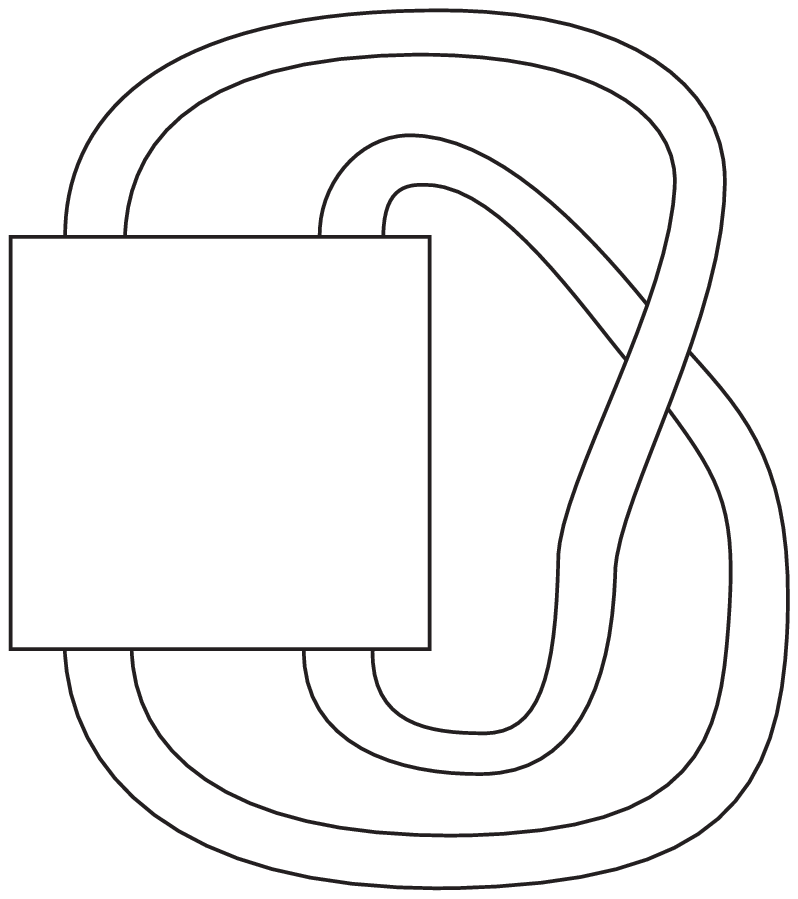} {.350}}

\def\khovanov{\pic{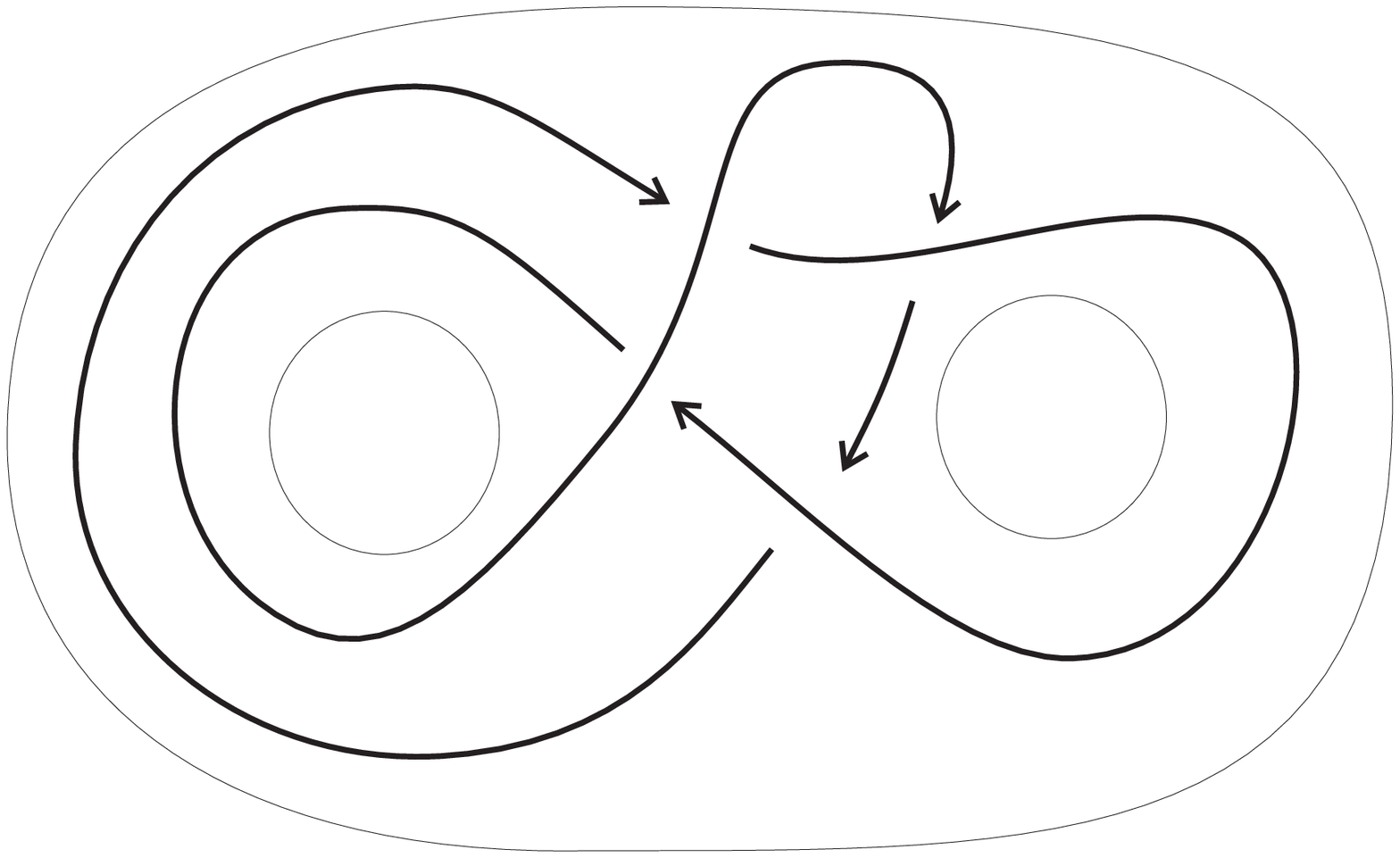} {.250}}
\def\khovtangle{\pic{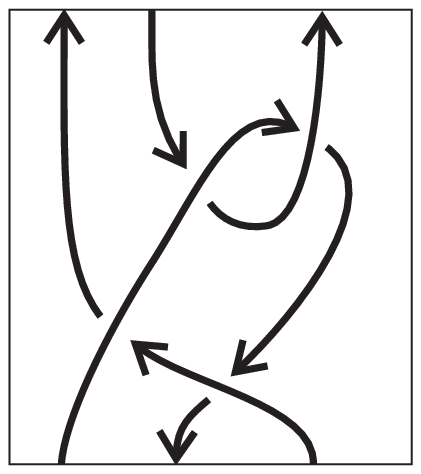} {.350}}
\def\Ktangle{\pic{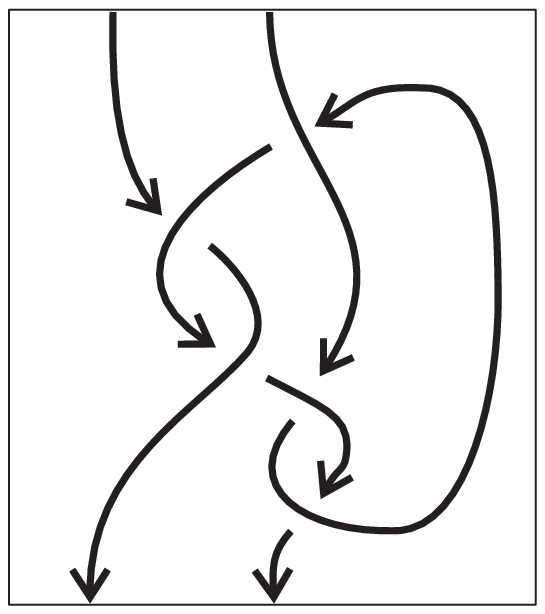} {.30}}
\def\DGtangle{\pic{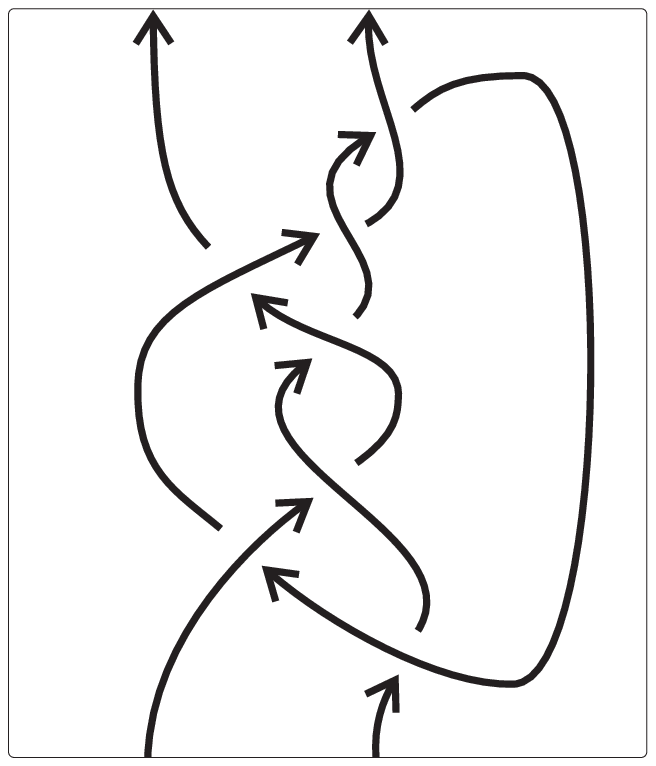} {.30}}
\def\Ttangle{\pic{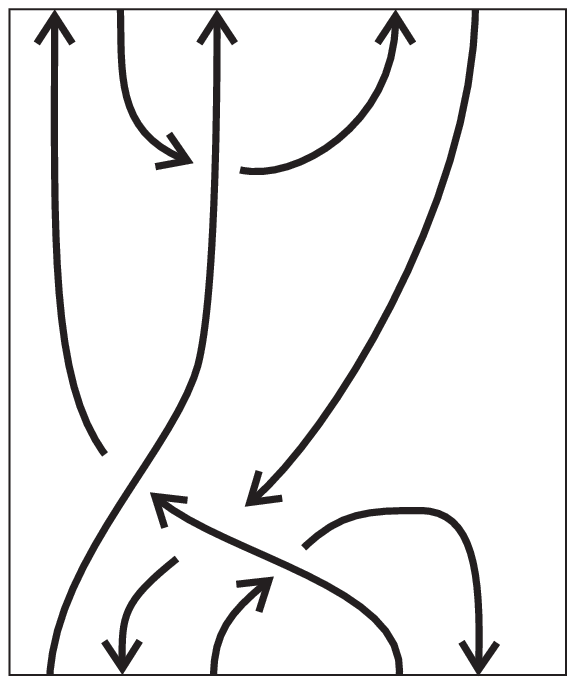} {.250}}

\def\Fhandlebody{\pic{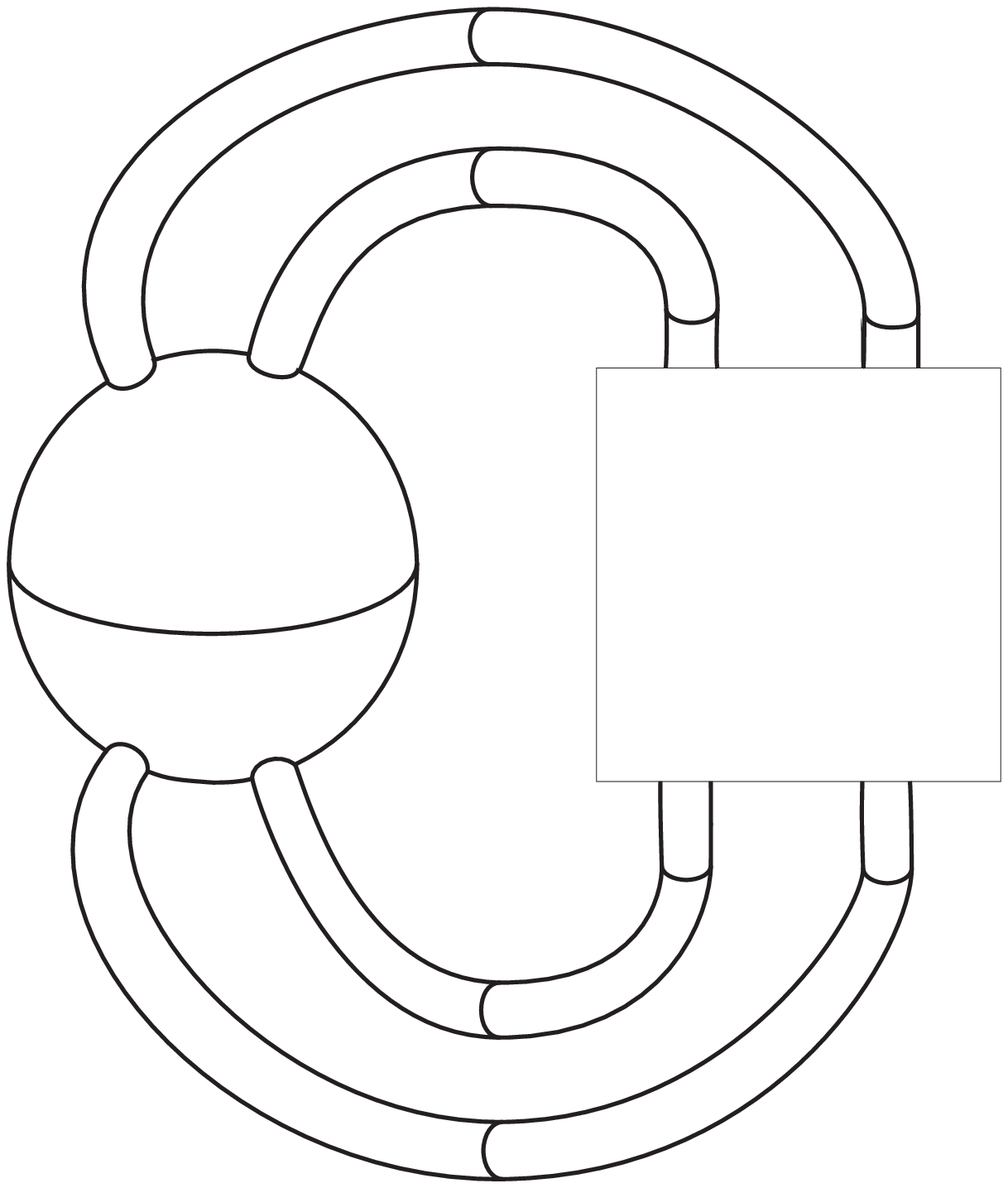} {.30}}

\def\Whandlebody{\pic{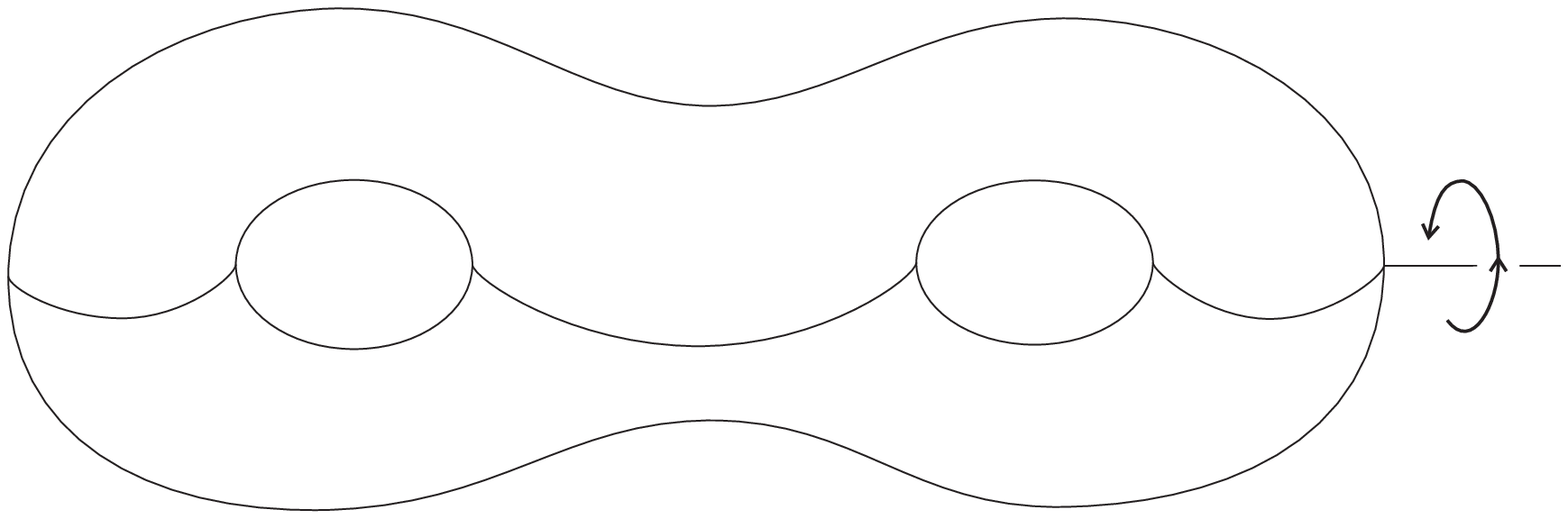} {.40}}

\def\Fsurface{\pic{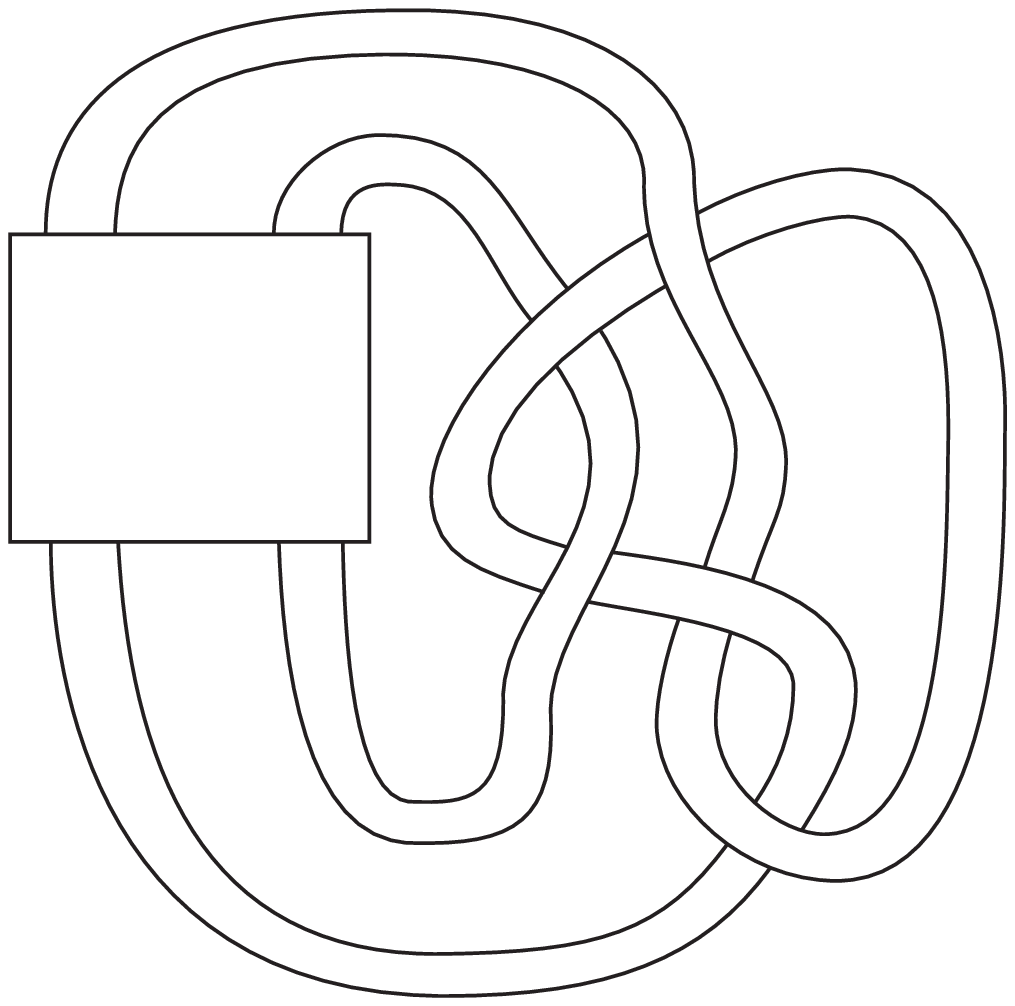} {.30}}

\def\CC{\mathcal {C}}

\newcommand{\bc}{\begin{center}}
\newcommand{\ec}{\end{center}}

\newcommand{\be}{\begin{equation}}
\newcommand{\ee}{\end{equation}}

\newcommand{\x}{\times}

\newtheorem{theorem}{Theorem}

\newenvironment{remark}{\par\smallskip%
\noindent\textbf{Remark.}\  }%
{\par\smallskip}
{\par\smallskip}
\renewenvironment{proof}[1][Proof]{\textit{#1.} }{\hfill \rule{0.5em}{0.5em}}

\begin{document}

\bc{\Large\bf Invariants of genus $2$ mutants\\[3mm]}
{\sc H. R. Morton {\rm and } N. Ryder\\[2mm]}
 {\small \sl Department of Mathematical Sciences\\ University of Liverpool\\ Peach Street, Liverpool L69 7ZL 
}
\ec
\begin{abstract} Pairs of genus $2$ mutant knots can have different Homfly polynomials, for example some $3$-string satellites of Conway mutant pairs. We give examples which have different Kauffman $2$-variable polynomials, answering a question raised by Dunfield et al in their study of genus 2 mutants.  While pairs of genus $2$ mutant knots have the same Jones polynomial, given from the Homfly polynomial by setting $v=s^2$, we give examples whose Homfly polynomials differ when $v=s^3$. We also give examples which differ in a Vassiliev invariant of degree $7$, in contrast to satellites of Conway mutant knots.
\end{abstract}

\section{Introduction}
Genus $2$ mutation of knots was introduced by   Ruberman \cite{Ruberman} in a general 3-manifold. Cooper and Lickorish \cite{CL} give a nice account of an equivalent construction for knots in $S^3$, using genus $2$ handlebodies, and it is this construction that we shall use here.

Genus $2$ mutant knots provide a test-bed for comparing knot invariants, in the sense that they can be shown to share a certain collection of invariants, and so any invariant on which some mutant pair differs must be completely independent of the shared collection. This procedure can be refined by restricting further the class of genus $2$ mutants under consideration, so as to increase the shared collection, and then looking for invariants which differ on some restricted mutants.

In a recent paper \cite{DGST} Dunfield, Garoufalidis, Shumakovitch and Thistlethwaite survey some of the known results about shared invariants for genus $2$ mutants, and show that Khovanov homology is not shared in general. They also give an example of a pair of genus $2$ mutants with 75 crossings which differ on their Homfly polynomial. These are smaller examples than the known satellites of the Conway and Kinoshita-Teresaka knots \cite{MC}.  They ask for examples of genus $2$ mutants which don't share the 2-variable Kauffman polynomial, in the expectation that their $75$ crossing knots, which are out of range of current programs for calculating the Kauffman polynomial, will indeed give such an example.

In this paper we give a number of smaller genus $2$ mutant pairs with different Homfly polynomials, and show that they also have different 2-variable Kauffman polynomials. The smallest examples to date, shown in figure \ref{ourexample}, have $55$ crossings. The fact that   their Kauffman polynomials are different  can be detected without having to make a complete calculation. The difference in their Homfly polynomials persists in this example, and in some but not all of the other examples, after making the substitution $v=s^3$. This substitution  calculates their quantum $sl(3)$ invariant when coloured by the fundamental $3$-dimensional module.

We note too a distinction between general genus $2$ mutants and those arising as satellites of Conway mutant knots, by exhibiting examples of  a pair of genus $2$ mutants which differ on a degree $7$ Vassiliev invariant, while work of Duzhin \cite{Duzhin} ensures that satellites of Conway mutants share all Vassiliev invariants of degree $\le 8$, extended to degree $10$ more recently by Jun Murakami \cite{Jun}.

\section{The general setting}
The satellite knot $K*Q$ of a framed oriented knot $K$ is  constructed, as a framed oriented knot, by taking a framed oriented curve $Q$ in the standard solid torus $V$.  Embed $V$ in ${\bf R}^3$ by following the knot $K$, using the embedding $h:V\to {\bf R}^3$ defined by regarding $V$ as a thickened annulus and carrying the annulus to the framing annulus of $K$. Then $K*Q$ is the curve $h(Q)\subset {\bf R}^3$, with the induced orientation and framing.

In the illustration in figure \ref{figsatellite} the framing of each curve is given implicitly by the blackboard framing.
\begin{figure}[ht]
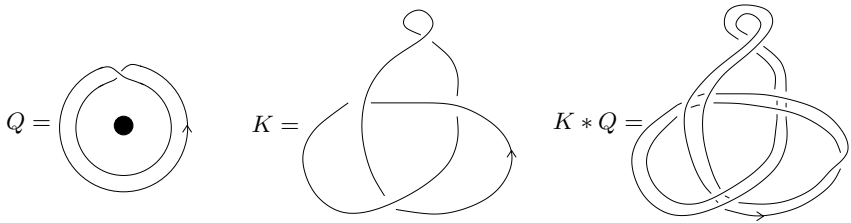

\labellist
\small
\pinlabel {$Q=$} at -60 195
\pinlabel {$K=$} at 430 195
\pinlabel {$K*Q=$} at 1070 195
\endlabellist
\bc
\Satellite
\ec
\caption{Satellite construction} \label{figsatellite}
\end{figure}

We can make
a similar construction, starting from a framed  oriented curve $P$ in the standard genus $2$ handlebody $W$.

The $\pi$-rotation $\tau:W\to W$, illustrated in figure \ref{hyperelliptic},   has $6$ fixed points on $\partial W$, where it restricts to the {\em hyperelliptic involution} with quotient $S^2$. This lies in the centre of the mapping class group of $\partial W$ and is unique up to conjugation by a homeomorphism isotopic to the identity.  

\begin{figure}[ht!]
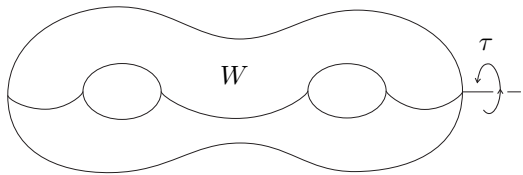

\bc {
\labellist
\pinlabel{$W$} at 281 645 
\pinlabel{$\tau$} at 524 675 
\endlabellist
\Whandlebody}
\ec
\caption{The rotation $\tau$}\label{hyperelliptic}
\end{figure}

Apply $\tau$ to $P$ to get another curve $\tau(P)\subset W$.  
For any embedding  $h:W\to {\bf R}^3$ the pair of knots
 $h(P)$ and $h(\tau(P))$ are called {\em genus $2$ mutants}.

\subsection{Satellites of genus $2$ mutants}
\begin{theorem}\label{g2satellite}
Satellites of genus $2$ mutants are themselves genus $2$ mutants.
\end{theorem}
\begin{proof}
The satellite $h(P)*Q$ of the framed knot $h(P)$ using a pattern $Q$ in the thickened annulus  is the same as the knot constructed by taking the satellite $P*Q$ in $W$ of the curve $P$ and then applying $h$, since the framings correspond. Then $h(P)*Q=h(P*Q)$. Similarly   $h(\tau(P))*Q=h(\tau(P)*Q)=h(\tau(P*Q))$ with the matching framing and orientation. Hence the satellites $h(P)*Q$ and $h(\tau(P))*Q$ of the genus $2$ mutants $h(P)$ and $h(\tau(P))$ are genus $2$ mutants.
\end{proof}

It is   easy to establish that genus $2$ mutants have the same Jones polynomial, using essentially the argument of Morton and Traczyk \cite{MT} in establishing that satellites of Conway mutants have the same Jones polynomial.

This argument is given directly in \cite{CL} and \cite{DGST} but we repeat it here for comparison with our extensions to some of the Homfly cases.

\begin{theorem} Genus 2 mutants have the same Jones polynomial.
\end{theorem}

\begin{proof} It is enough to work with the Kauffman bracket, defined by the usual skein relations
\bc
$\xunor\quad=\quad A\iunor\quad+\ A^{-1}\ \infunor$ \ , \quad
$\iddeltaunor\quad=\  -(A^2+A^{-2})\ \idunor$\ .
\ec
 We can treat a framed curve $P$ in $W$ as an element in the Kauffman bracket skein of a surface $S$ with $W\cong S\times I$ when calculating the Kauffman bracket of the genus $2$ mutants $h(P)$ and $h(\tau(P))$.  We   take $S$ to be a disc with $2$ holes. The involution $\tau$ on $W$ is induced by the involution on $S$ which preserves the boundary components.

The Kauffman bracket skein of $S$ is spanned by diagrams in $S$ without crossings or null-homotopic curves. Such diagrams consist of unoriented curves parallel to the boundary components and are hence all unchanged by the involution $\tau$ on $S$. Then $\tau(P)=P$ as elements of the skein of $S$, and so $h(\tau(P))=h(P)$ as   elements of the skein of the plane. Since any diagram $K$ in the plane represents  $<K>\unknotunor$ in the skein of the plane, where $<K>$ is the Kauffman bracket of $K$, it follows  that the genus $2$ mutants $h(\tau(P))$ and $h(P)$ have the same Kauffman bracket.
\end{proof}

Theorem \ref{g2satellite} then shows that genus $2$ mutants share all their satellite Jones invariants.

\subsection{Genus $2$ embeddings following a $2$-tangle}

We now show how to use a framed oriented $2$-tangle $F$ to define an embedding $h:W\to {\bf R}^3$ in such a way that we can readily compare the framed curves  $h(P)$ and $h(\tau(P))$. This embedding is said to \emph{follow} the tangle $F$.

Attaching the two thickened arcs of $F$ to a solid ball   results in a genus $2$ handlebody as in figure \ref{Fhandlebody}
which is to be the image of $h$.

\begin{figure}[ht!]
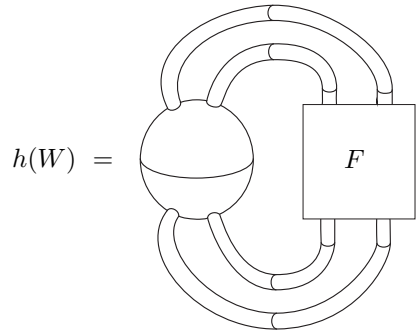

\bc $h(W)\ =\ $
{
\labellist
\pinlabel{$F$} at 394 472 
\endlabellist
\Fhandlebody}
\ec
\caption{The handlebody following a tangle $F$}\label{Fhandlebody}
\end{figure}

To specify $h$ we   assume that $F$ has a framing, in other words each arc has a specified ribbon neighbourhood.
Define a surface $S_F$ in ${\bf R}^3$ consisting of a square plus two ribbons following the framing of $F$, illustrated in figure \ref{Fsurface} using the tangle $F$ from figure \ref{mutant}.

\begin{figure}[ht!]
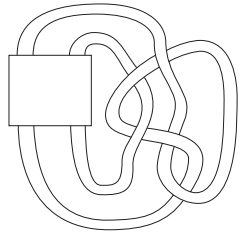

\bc {
\labellist
\endlabellist
\Fsurface}
\ec
\caption{The surface following a framed tangle}\label{Fsurface}
\end{figure}

Regard $W$ as the thickening, $S\x I$, of  a standard surface $S$,   and define $h$ by thickening a map from $S$ to $S_F$. 
Our choice of $S$, and hence the description of $h$,    depends on the nature of the tangle $F$.
We distinguish two types of oriented $2$-tangle:
\begin{enumerate}
\item A {\em pure} tangle, where the arcs join the two bottom points to the corresponding top points on the same side.
\item A {\em transposing} tangle, where the arcs join the two bottom points to the top points on opposite sides.
\end{enumerate}

\begin{remark} The terms \emph{parallel} and \emph{diagonal} are used in \cite{MT} for the connections in these two types of tangle.
\end{remark}

1. When $F$ is a \emph{pure} tangle the surface $S_F$ is a disc with $2$ holes. Take $S$  to be the square with two ribbons in figure \ref{squarepants} and map $S$ to $S_F$ by taking the square to the square, and the two ribbons to the ribbons around the arcs of $F$. 
\begin{figure}[ht!]
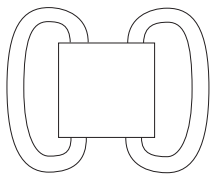

\bc
\squarepants
\ec\caption{The disc with $2$ holes}\label{squarepants}
\end{figure}

2. When $F$ is a \emph{transposing} tangle the surface $S_F$ is a torus with one hole.  Take $S$ to be the square with two ribbons in figure \ref{squaretorus} and again map $S$ to $S_F$ by mapping the square to the square, and the ribbons around the arcs of $F$.

\begin{figure}[ht!]
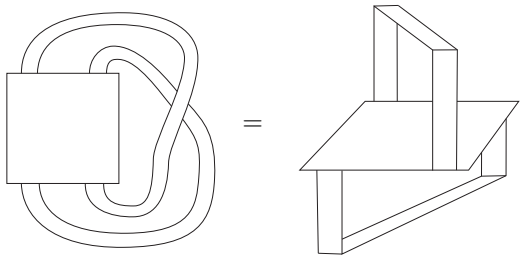

\bc
\torushole \quad = \quad \squaretorus
\ec\caption{The torus with one hole}\label{squaretorus}
\end{figure}

We   say that $h$ has been constructed by {\em following} the tangle $F$. An embedded handlebody in ${\bf R}^3$ always arises by following some tangle $F$,  although the choice of $F$ is not unique.

We can get a good view of the pair of mutants constructed from a curve $P\subset W$ by following a tangle $F$.
The map $\tau:W\to W$ is a thickened map from $S$ to $S$, which maps the square and each ribbon to itself.  In case 1, $\tau$ is $\pi$-rotation about the horizontal $x$-axis, which we write as $\tau_1$ when restricted to the square.  In case 2, $\tau$ is $\pi$-rotation about the $z$-axis orthogonal to the plane of the square, and we write $\tau_2$ for this rotation restricted to the square. These rotations are indicated in figure \ref{figrotations}.

\begin{figure}[ht!]
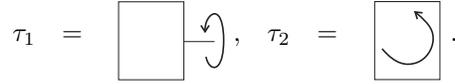

\bc
$\tau_1$\quad {=} \quad {\labellist
%\pinlabel{$T$} at 96 668
\endlabellist
\mutanttwo }\ ,\quad  $\tau_2$\quad {=} \quad{\labellist
%\pinlabel{$T$} at 102 629
\endlabellist
\mutantthree}\ .
\ec
\caption{Rotations of the square}\label{figrotations}
\end{figure}

Draw $P$ itself as a diagram on the surface $S$, so that its framing is the blackboard framing from $S$. We can assume that $P$ runs through each ribbon of $S$ in a number of parallel curves, possibly with different orientations. Suppose that there are $m_1$ curves in one ribbon and $m_2$ in the second, numbered from the attachment to the top edge of the square. The rest of the curve $P$ determines a framed $m$-tangle $T$ in the square, with $m=m_1+m_2$. 

In the case of a pure tangle $F$ the knot $h(P)$ has a diagram as shown in figure \ref{parallels},  where $F^{(m_2,m_1)}$ is the $(m_2,m_1)$ parallel of the framed tangle $F$ with appropriate orientations, and the tangle $T$ lies in the square. The mutant knot $\tau(h(P))$ has $\tau_1(T)$ in place of $T$, with all orientations in $F^{(m_2,m_1)}$  reversed.

\begin{figure}[ht!]
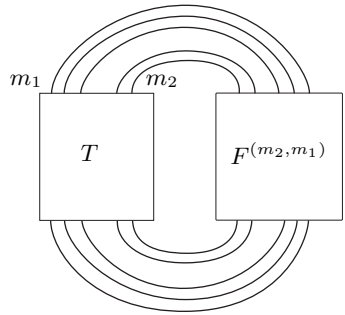

\bc
{
\labellist
\small
\pinlabel {$T$} at 136 420
\pinlabel {$F^{(m_2,m_1)}$} at 421 420
\pinlabel {$m_1$} at 40 529
\pinlabel {$m_2$} at 245 529
\endlabellist}
\parallels
\ec
\caption{The diagram for a knot following a pure tangle $F$}\label{parallels}
\end{figure}

In the case of a transposing tangle the diagram is shown in figure \ref{crossparallels}, where $\tau(h(P))$ now has $\tau_2(T)$ in place of $T$.

\begin{figure}[ht!]
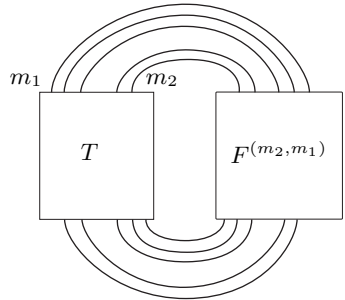

\bc
{
\labellist
\small
\pinlabel {$T$} at 136 420
\pinlabel {$F^{(m_2,m_1)}$} at 421 420
\pinlabel {$m_1$} at 40 529
\pinlabel {$m_2$} at 245 529
\endlabellist}
\crossparallels
\ec
\caption{The diagram for a knot following a transposing tangle $F$}\label{crossparallels}
\end{figure}

\subsection{Conway mutants}
 For an oriented tangle $T$ write $\tau_1(T)$ and $\tau_2(T)$ for the $\pi$-rotations of $T$ about the $x$-axis and $z$-axis respectively, as used above. Then $\tau_3(T)=\tau_1\tau_2(T)$ is the $\pi$-rotation of $T$ about the $y$-axis, so that 

\bc
$\tau_1(T)$\quad {=} \quad {\labellist
\pinlabel{$T$} at 96 668
\endlabellist
\mutanttwo }\ ,\quad  $\tau_2(T)$\quad {=} \quad{\labellist
\pinlabel{$T$} at 102 629
\endlabellist
\mutantthree}\ ,\quad $\tau_3(T)$\quad {=} \quad
{\labellist
\pinlabel{$T$} at 129 709
\endlabellist
\mutantone }\ ,
\ec

 The term \emph{mutant} was coined by Conway, and  refers to the following
general construction.

Suppose that a knot $K$ can be decomposed into 
two oriented  $2$-tangles $F$ and $G$ as in figure \ref{mutant}. 
Any knot $K'$
 formed by replacing the  tangle $F$ with the tangle
$F'=\tau_i(F), i=1,2,3$, 
reversing its string orientations if
necessary  is called a 
 \emph{(Conway) mutant} of $K$. 
\begin{figure}[ht!]
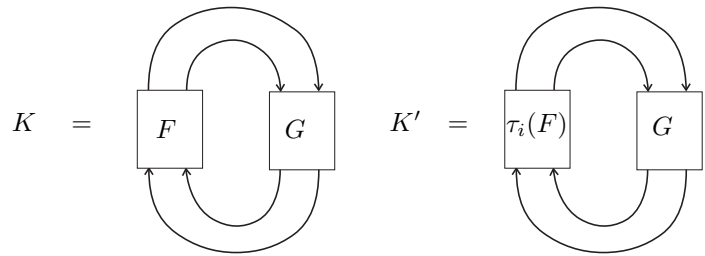

\begin{center}
$K$ \quad {=} \quad
{
\labellist
\pinlabel{$F$} at 100 733
\pinlabel{$G$} at 181 733
\endlabellist
\mutantbox}\qquad
$K'$\quad {=} \quad{\labellist
\pinlabel{{$\tau_i(F)$}} at 102 735
\pinlabel{{$G$}} at 181 735
\endlabellist
\mutantbox}
\end{center}
\caption{A knot with mutants}\label{mutant}
\end{figure}

The two  $11$-crossing knots  in figure \ref{CKT},
 found by 
Conway and  Kinoshita-Teresaka, are the best-known example of a pair of mutant knots. 

\begin{figure}[ht!]
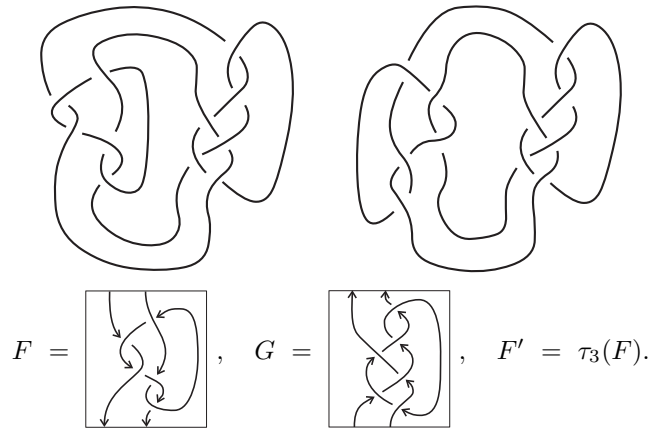

\begin{center}{  \Conway\  \qquad
 \KT\ }\\[2mm]
$F\ =\ \Ktangle\ , \quad G\ =\ \Ctangle\ , \quad F'\ =\ \tau_3(F).$
\end{center}
\caption{The Conway and Kinoshita-Teresaka mutant pair, and their constituent tangles}\label{CKT}
\end{figure}

\subsection{Conway mutants as genus $2$ mutants}

Any knot $K$ made up of two $2$-tangles $F$ and $G$ as in figure \ref{mutant} lies in two genus $2$ handlebodies, one following $F$ and the other following $G$. 
Each of these handlebodies defines a genus $2$ mutant of $K$. We call them $K_F$ and $K_G$ respectively.   

Since $K$ is a {\em knot}, one of the tangles $F,G$ is pure and the other is transposing. Let us suppose that $F$ is pure. Then $K_F$ and $K_G$ have diagrams as shown in  figure \ref{mutantrev}.

\begin{figure}[ht!]
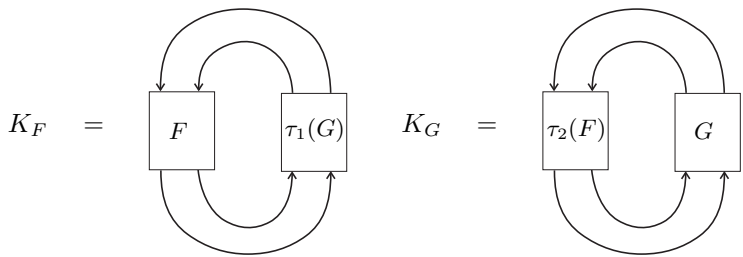

\begin{center}
$K_F$ \quad {=} \quad
{
\labellist
\small
\pinlabel{$F$} at 100 733
\pinlabel{$\tau_1(G)$} at 186 733
\endlabellist
\mutantboxrev}\qquad
$K_G$ \quad {=} \quad
{
\labellist
\small
\pinlabel{$\tau_2(F)$} at 103 733
\pinlabel{$G$} at 183 733
\endlabellist
\mutantboxrev}

\end{center}
\caption{Genus $2$ mutants of $K$}\label{mutantrev}
\end{figure}

We can repeat the construction on these knots.  $K_F$ lies in the handlebody following $\tau_1(G)$. Since $\tau_1(G)$ is transposing we get a genus $2$ mutant $K_{F \tau_1(G)}$. The same knot $K_{G \tau_2(F)}=K_{F \tau_1(G)}$ arises as a genus $2$ mutant of $K_G$ from the handlebody following $\tau_2(F)$, shown in figure \ref{doublemutant}.

\begin{figure}[ht!]
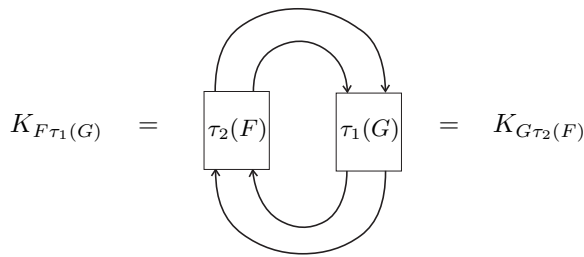

\begin{center}
$K_{F \tau_1(G)}$ \quad {=} \quad
{
\labellist
\small
\pinlabel{$\tau_2(F)$} at 103 733
\pinlabel{$\tau_1(G)$} at 186 733
\endlabellist
\mutantbox} \quad $=\quad K_{G \tau_2(F)}$
\end{center}
\caption{A further genus $2$ mutant, completing the Conway mutants of $K$}\label{doublemutant}
\end{figure}

Rotation of the diagrams of $K_F$ and $K_{F \tau_1(G)}$ about the $x$-axis shows that, up to a choice of string orientation, these three knots $K_F, K_G$ and $K_{F \tau_1(G)}$ are the three Conway mutants of $K$ given by replacing $F$ with $\tau_1(F), \tau_2(F)$ or $\tau_3(F)$ respectively.

It follows that satellites of Conway mutants, with this orientation convention, are related by genus $2$ mutation.

We have already seen that these must all share the same Jones polynomial. We 
now look at the Homfly polynomial of genus $2$ mutants.

\section{The Homfly polynomial of genus $2$ mutants}
We   use the framed version of the Homfly polynomial based on the skein relations
\bc $\Xor\quad -\quad\Yor \quad=\quad (s-s^{-1}) \Ior$\\[2mm]
$\Rcurlor\quad = \quad v^{-1}\ \Idor\ , \qquad \Lcurlor\quad=\quad v\  \Idor\  .$
\ec

The Homfly polynomial of a link in ${\bf R}^3$ is unchanged if the orientations of {\em all} its components are reversed.  The Homfly skein of the annulus $\CC$ is unchanged when the annulus is rotated by $\pi$, reversing its core orientation, and at the same time all string orientations are reversed. To compare the Homfly polynomials of two genus $2$ mutants $h(P)$ and $h(\tau(P))$, or indeed any satellite of them, it is enough to consider $h(\tau(P))$ with orientation reversed. 

Given a framed  oriented curve $P$ in $W$ we may then regard $W$ as the thickened surface $S$ which is the disc with 2 holes in figure \ref{squarepants}, and compare $P$ with $\tau(P)$ after reversing the orientation of $\tau(P)$.  If we can present $P$ as an $(m_1+m_2)$-tangle in the square with $m_1$ and $m_2$ curves following the two ribbons then we can write $P$ in the skein of the twice-punctured disc $S$ as a linear combination of simpler curves, each presented by a tangle with at most this number of curves in the ribbons.

 Even if our curve $P$ has originally been drawn in a picture following a transposing tangle, with $m_1$ and $m_2$ curves around the ribbons there, it can be redrawn as a curve following a pure tangle with the same numbers $m_1$ and $m_2$.

The first observation is that if $m_1=m_2=1$ then the genus $2$ mutants are Conway mutants, and their Homfly polynomials agree. This is because any $2$-tangle can be reduced to a linear combination of $2$-tangles which are unchanged under $\tau_1$ plus string orientation reversal.

In the case $m_1,m_2\le2$  the curve $P$ again reduces in the skein of $S$ to a combination of curves in the skein of $S$ which are again unchanged by the rotation $\tau$ with reversal of string orientation. This is essentially the result of Lickorish and Lipson \cite{LL}. There are a couple of cases depending on the relative orientation of the curves in the two ribbons. This argument then covers the case of any $2$-string satellite of a pair of Conway mutants, as these can be presented as genus $2$ mutants with $m_1=m_2=2$.

The existence of $3$-string satellite knots around the Conway and Kinoshita-Teresaka mutant pair with different Homfly polynomials, described in detail in \cite{MC}, following the earlier calculations by Morton and Traczyk, shows that there are some genus $2$ mutants with $m_1=m_2=3$, constructed by following the constituent tangle $G$ in figure \ref{mutant}, which have different Homfly polynomials. Take, for example, the tangle $T$ to be the $3$-parallel $F^{(3,3)}$ of the tangle $F$ in figure \ref{mutant} composed with the braid $\sigma_1\sigma_2$ and follow the tangle $G$ to give a knot with $101$ crossings. This is in fact a satellite of the Conway knot, whose genus $2$ mutant has $\tau_2(T)$ in place of $T$.

\subsection{Genus $2$ mutants with different Kauffman polynomials}
In \cite{DGST} the authors exhibit a pair of genus $2$ mutants with $75$ crossings, which have different Homfly polynomials, and they ask whether genus $2$ mutants can have different Kauffman polynomials. Although   confident that this is the case they were unable to calculate the polynomials for their $75$ crossing example, constructed following the pure $7$-crossing tangle $DG$ shown in figure \ref{DG}.

\begin{figure}[ht!]
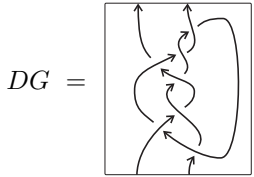

\bc
$DG \ =\ $
\DGtangle
\ec
\caption{The $7$-crossing tangle $DG$}\label{DG}
\end{figure}

We give here a number of examples of genus $2$ mutants with different Kauffman polynomials.

\begin{theorem}\label{g2kauffman}
The genus $2$ mutant pair of knots constructed by following the tangle $DG$, with $m_1=m_2=3$, using the $6$-string positive permutation braid $B=\sigma_1\sigma_2\sigma_1\sigma_3\sigma_2\sigma_4\sigma_3\sigma_5\sigma_4$, shown in figure~\ref{braid}, or its reverse $\tau_1(B)$ as the tangle $T$, have different Kauffman polynomials.
\end{theorem}

\begin{figure}[ht!]
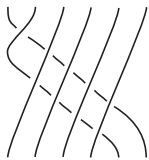

\begin{center}
\Braid
\end{center}
\caption{The braid $B=\sigma_1\sigma_2\sigma_1\sigma_3\sigma_2\sigma_4\sigma_3\sigma_5\sigma_4$}
\label{braid}
\end{figure}

\begin{proof}
The two knots are presented as closed $9$-braids with $72$ crossings, so it is quite easy to calculate their Homfly polynomials using the Morton-Short program \cite{MS} based on the Hecke algebras. When these are compared, as polynomials in $z=s-s^{-1}$ with coefficients in ${\bf Z}[v^{\pm 1}]$ they can be seen to differ in their constant term $P_0(v)$.  Now Lickorish shows in \cite{L} that $P_0(v)$ is also the constant term of the Kauffman polynomial when expanded similarly, and hence the Kauffman polynomials of the two knots are different.
\end{proof}
\begin{remark}This argument could not have been used for the $75$ crossing knots in \cite{DGST}, since their Homfly polynomials have the same constant term $P_0(v)$.
\end{remark}
\subsection{Vassiliev invariants}
We compared the Vassiliev invariants of the genus $2$ mutants, by expanding the difference of their Homfly polynomials as a power series in $h$
taking $s=e^{ h/2}$ and $v=s^N=e^{Nh/2}$. In the $75$ crossing examples from \cite{DGST}  the  lowest degree term of the difference is
\[N \left( N-1 \right)  \left( N-2 \right)  \left( N+2 \right)\left( N+1 \right)  \left( 13\,{N}^{2}+51 \right)h^{11}, \] while for our $72$ crossing example it is
\[3N(N-1)(N-2)(N-3)(N+3)(N+2)(N+1)h^7.\]
This shows that the $72$ crossing knots differ in a Vassiliev invariant of degree at most $7$.
Consequently satellites of Conway mutants share more Vassiliev invariants than general genus $2$ mutants, since  they have all Vassiliev invariants of degree $\le 10$ in common, using the result from \cite{MC} that Vassiliev invariants of degree $\le k$ of  a satellite $K*Q$  are Vassiliev invariants of $K$ of the same degree, and Jun Murakami's result \cite{Jun} about Vassiliev invariants of Conway mutants. 

\subsection{The Homfly invariants with $v=s^3$}
In our $72$ crossing examples the string orientations around each ribbon are all in the same sense 
 $+++$, and as a result the knots have the same  Homfly invariant after the substitution $v=s^3$. This is  a general consequence of the analysis of the Kuperberg skein of the surface $S$ in \cite{MR} for the case $m_1=m_2=3$ in which all the orientations  around the ribbons are $+$.

In contrast the $75$ crossing examples in \cite{DGST} use a $6$-tangle $T$, again with $m_1=m_2=3$, where the orientations of the three strands around one of the ribbons are $++-$ while around the other they are $+++$. In this case the Homfly polynomials remain different
when $v=s^3$.
The difference,  as a Laurent polynomial in $s$, is:
\begin{eqnarray*}
&{}& {s}^{-28}\left( {s}^{4}-{s}^{2}+1 \right)  \left( {s}^{4}+{s}^{3}+{s}
^{2}+s+1 \right)  \left( {s}^{4}-{s}^{3}+{s}^{2}-s+1 \right) \\
&{}& \ \ \ \left( {s}^{8}+1 \right)  \left( {s}^{6}+{s}^{5}+{s}^{4}+{s}^{3}+{s}^{2}+s+1
 \right)  \left( {s}^{6}-{s}^{5}+{s}^{4}-{s}^{3}+{s}^{2}-s+1 \right)\\
&{}& \ \ \ \left( {s}^{2}-s+1 \right) ^{2} \left( {s}^{2}+s+1 \right) ^{2}
 \left( {s}^{4}+1 \right) ^{2} \left( {s}^{2}+1 \right) ^{3} \left( s-
1 \right) ^{11} \left( s+1 \right) ^{11}
\end{eqnarray*} 

We had originally tried to make use of the difference when $v=s^3$ of the $75$-crossing examples to show that the Kauffman polynomials are also different. We  planned to argue through the comparison of the Homfly polynomials of a certain $2$-string satellite at $v=s^4$, without actually calculating this Homfly polynomial, which would be well out of range. Our aim was to make use of a comparison in \cite{Morton} between this evaluation of the satellite invariant and a different evaluation of the Kauffman polynomial of the original knots, knowing something of the evaluations of the satellite invariant at $v=s^3$.  Unfortunately the difference in the invariants at $v=s^3$ contains a factor $ \left( {s}^{6}+{s}^{5}+{s}^{4}+{s}^{3}+{s}^{2}+s+1
 \right)  $ which means that the agreement of the evaluations of the satellite at $v=s^4$ can not be excluded.

This has also proved to be the case in any other examples that we have found where the evaluations at $v=s^3$ are different, so there may be some underlying reason behind this in general.

\subsection{Smaller examples}
Inspired by the combinatorial interpretations of the $v=s^3$ substitution in leading to the Kuperberg skein of the twice-punctured disc we have found a pair of  examples following $DG$ with $m_1=3,m_2=2$ and orientations $++-$ and $+-$.  The curve $P$ is shown in figure \ref{Phandlebody} as a diagram in the disc with two holes, $S$, along with the resulting $5$-tangle $T$.  
\begin{figure}[ht!]
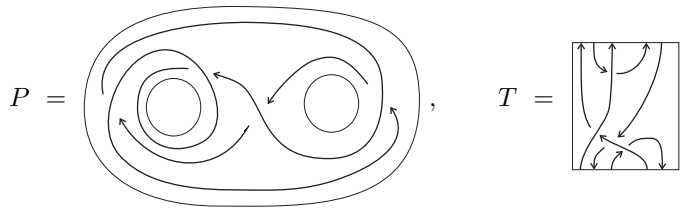

\begin{center}
$P\  = \ $
\twicepunc\ , \qquad $T\  = \ $
\Ttangle 
\end{center}
\caption{The curve $P$ in the standard handlebody, and related tangle $T$}\label{Phandlebody}
\end{figure}

We construct two $55$-crossing  genus $2$ mutants    from $P$ by following the tangle $DG$, to give 
the knot $S_{55}$, shown in figure \ref{ourexample}.  Its mutant partner $S^{'}_{55}$ is given by applying the rotation $\tau_1$ to the tangle $T$.
\begin{figure}[ht!]
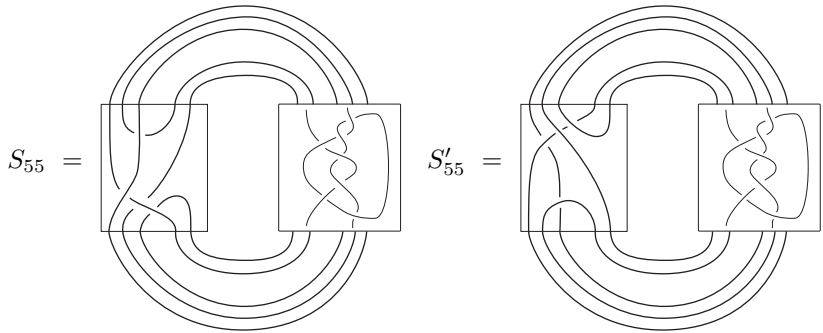

\begin{center}
$S_{55}\  = \ $
\Sfiftyfive\quad $S'_{55}\  = \ $
\Tfiftyfive 
\end{center}
\caption{Two $55$-crossing genus $2$ mutants with different Homfly and Kauffman polynomials}\label{ourexample}
\end{figure}
\begin{theorem} \label{S55}
The knots $S_{55}$ and $S'_{55}$ shown in figure \ref{ourexample} have different Homfly and Kauffman polynomials. Their Homfly polynomials still differ after the substitution $v=s^3$. 
\end{theorem}
\begin{proof}
The coefficients for the Homfly polynomials of $S_{55}$ and $S^{'}_{55}$ are shown below. They were calculated using Ochiai's program \cite{Ochiai}, since the knots are not readily expressed as closed braids.  In the table the Lickorish-Millett variables $l$ and $m$ are used, with $l^2=-v^2$ and $m^2=-z^2$.
\[
\begin{array}{|c|rrrrrrrrr|}
\hline
S_{55} & l^{-4} & l^{-2} & 1 & l^2 & l^4 & l^6 & l^8 & l^{10} & l^{12} \\
\hline
1 & -36 & -122 & -143 & -67 & -23 & -32 & -23 & -5 & \\
m^2 & 276 & 986 & 1199 & 550 & 148 & 223 & 172 & 34 & -3 \\
m^4 & -757 & -3003 & -3884 & -1811 & -345 & -567 & -478 & -75 & 20 \\
m^6 & 1048 & 4688 & 6531 & 3158 & 400 & 718 & 690 & 76 & -45 \\
m^8 & -827 & -4243 & -6360 & -3217 & -253 & -499 & -585 & -39 & 34 \\
m^{10} & 388 & 2355 & 3774 & 1985 & 87 & 192 & 302 & 10 & -10 \\
m^{12} & -107 & -814 & 1386 & -746 & -15 & -38 & -92 & -1 & 1 \\
m^{14} & 16 & 171 & 308 & 166 & 1 & 3 & 15 & & \\
m^{16} & -1 & -20 & -38 & -20 & & & -1 & & \\
m^{18} & & 1 & 2 & 1 & & & & & \\
\hline
\end{array}
\]
\[
\begin{array}{|c|rrrrrrrrr|}
\hline
S^{'}_{55} & l^{-4} & l^{-2} & 1 & l^2 & l^4 & l^6 & l^8 & l^{10} & l^{12} \\
\hline
1 & -38 & -135 & -178 & -116 & -58 & -39 & -16 & & 1 \\
m^2 & 257 & 924 & 1171 & 662 & 288 & 209 & 60 & -34 & -16 \\
m^4 & -687 & -2591 & -3205 & -1587 & -562 & -448 & -72 & 142 & 54 \\
m^6 & 964 & 3913 & 4779 & 2080 & 566 & 509 & 24 & -226 & -73 \\
m^8 & -782 & -3530 & -4260 & -1623 & -319 & -334 & 10 & 172 & 43 \\
m^{10} & 377 & 1991 & 2356 & 766 & 100 & 126 & -7 & -67 & -11 \\
m^{12} & -106 & -709 & -814 & -213 & -16 & -25 & 1 & 13 & 1 \\
m^{14} & 16 & 155 & 171 & 32 & 1 & 2 & & -1 & \\
m^{16} & -1 & -19 & -20 & -2 & & & & & \\
m^{18} & & 1 & 1 & & & & & & \\
\hline
\end{array}
\]
Immediately we can see that they have different Homfly polynomials. The first row of coefficients in each array is equivalent to $P_0(v)$, and so   the result of Lickorish shows that $S_{55}$ and $S^{'}_{55}$ must also have different Kauffman polynomials.

We obtain Vassiliev invariants as the coefficients of powers of $h$ in the power series given substituting $m = i(e^{\frac{h}{2}}-e^{-\frac{h}{2}})$, $l = ie^{\frac{Nh}{2}}$. The lowest term in the difference of the power series  for $S_{55}$ and $S^{'}_{55}$ is
\[3N(N-1)(N-2)(N-3)(N+3)(N+2)(N+1)h^7,\] so again these differ in a Vassiliev invariant of degree at most $7$. 
We can also look at $sl(3)$ invariant information as a Laurent polynomial in $s$ by making the substitutions $m = i(s-s^{-1})$, $l = is^3$. The difference is: 
\begin{eqnarray*}
&{}&s^{-24}\left( {s}^{4}-{s}^{2}+1 \right)  \left( {s}^{4}+{s}^{3}+{s}
^{2}+s+1 \right)  \left( {s}^{4}-{s}^{3}+{s}^{2}-s+1 \right) \left( {s}^{8}+1 \right) \\
&{}& \ \ \ \left( {s}^{6}+{s}^{5}+{s}^{4}+{s}^{3}+{s}^{2}+s+1
 \right)\left( {s}^{6}-{s}^{5}+{s}^{4}-{s}^{3}+{s}^{2}-s+1 \right) \\
&{}& \ \ \ \left( {s}^{2}+s+1 \right) ^{2} \left( {s}^{2}-s+1 \right) ^{2}
 \left( {s}^{4}+1 \right) ^{2} \left( {s}^{2}+1 \right) ^{3} \left( s-
1 \right) ^{8} \left( s+1 \right) ^{8}
\end{eqnarray*}

Here again there is a factor of $\left( {s}^{6}+{s}^{5}+{s}^{4}+{s}^{3}+{s}^{2}+s+1
 \right)$, as in the DGST case. The factor $(s-1)^8$ shows that they differ in a Vassiliev invariant of degree $8$ invariant arising from $sl(3)$.
\end{proof}

We have also constructed a pair of $56$-crossing genus $2$ mutants following the transposing Conway tangle $G$ with $6$ crossings, using the $6$-braid $\sigma_2\sigma_3$ and its rotation $\tau_2(\sigma_1\sigma_2)=\sigma_3\sigma_4$,  shown in figure \ref{braidexample}, with $m_1=m_2=3$. These are closed $9$-braids, closely related to the original more complicated Conway and Kinoshita-Teresaka satellites. Like our $72$-crossing examples  in theorem \ref{g2kauffman} this pair have different Kauffman polynomials, because of $P_0(v)$, and also differ in a degree $7$ Vassiliev invariant, but   share the same value when $v=s^3$.

\begin{figure}[ht!]
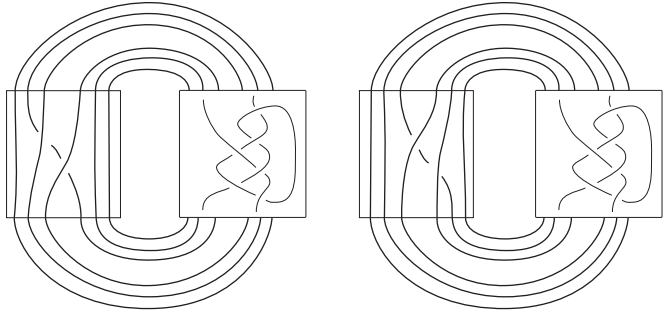

\begin{center}
\Sfiftysix\qquad 
\Tfiftysix 
\end{center}
\caption{Two closed $9$-braid genus $2$ mutants with different Homfly polynomial}\label{braidexample}
\end{figure}

\subsection{Other examples}
In \cite{DGST} there are several nice examples with $m_1=2, m_2=1$, following the pure tangle $AB$ in figure \ref{figone}, which have different Khovanov homology. 
\begin{figure}[ht!]
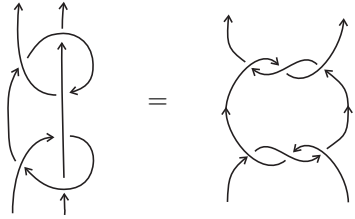

{\labellist
\small
\endlabellist
\bc
\AB\qquad=\qquad \ABreef
\ec}
\caption{The tangle $AB$ used in \cite{DGST} } \label{figone}
\end{figure}
 The simplest of these uses the curve $P$,  shown in figure \ref{khovanov} as a diagram in the disc with two holes, $S$, along with the resulting $3$-tangle $T$.  

\begin{figure}[ht!]
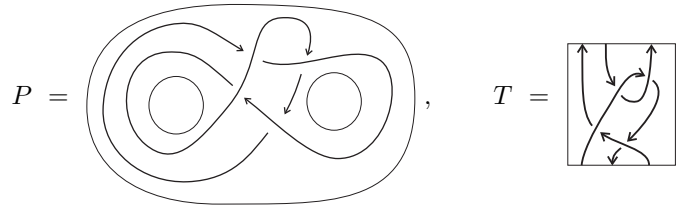

\begin{center}
$P\  = \ $
\khovanov\ , \qquad $T\  = \ $
\khovtangle 
\end{center}
\caption{A curve $P$, and related tangle $T$, giving mutants with different Khovanov homology}\label{khovanov}
\end{figure}

 It is interesting to speculate whether satellites of Conway mutant {\em knots} can ever have different Khovanov homology, given that they have a greater range of shared invariants than the general genus $2$ mutants.

There is a result of Wehrli \cite{Wehrli} giving two Conway mutant links with different Khovanov homology, but unlike Conway mutant \emph{knots} these two links are not   related by genus $2$ mutation.

\section*{Acknowledgements}

The first author would like to thank Prof. J.M.Montesinos and the Universidad Complutense, Madrid, for their hospitality and support during the preparation of this paper. The second author acknowledges the support of EPSRC under the doctoral training grant number EP/P500338/1.

\end{document}